\documentclass[a4paper,11pt]{article}
\usepackage{filecontents}
\usepackage{amssymb,amsthm,amsmath}
\usepackage[utf8]{inputenc}       
\usepackage{enumitem}
\usepackage[english]{babel}
\usepackage{mathtools}
\usepackage{capt-of}
\usepackage{url}
\usepackage{spreadtab}
\usepackage{array}
\usepackage{booktabs}
\usepackage{authblk}

\theoremstyle{definition}
\newtheorem{lemma}{Lemma}[section] 
\newtheorem{lause}[lemma]{Theorem}
\newtheorem{corollary}[lemma]{Corollary}
\newtheorem{remark}[lemma]{Remark}

\title{On the explicit upper and lower bounds for the number of zeros of the Selberg class}
\date{}
\author{Neea Paloj\"arvi \footnote{This work was supported by the Vilho, Yrj\"o and Kalle V\"ais\"al\"a Foundation of the Finnish Academy of Science and Letters.}}
\affil{Department of Mathematics and Statistics, \AA bo Akademi University, Domkyrkotorget 1, 20500 \AA bo, Finland, neea.palojarvi@abo.fi}
\begin{document}

\maketitle 
\section*{Abstract}

In this paper we prove explicit upper and lower bounds for the error term in the Riemann-von Mangoldt type formula for the number of zeros inside the critical strip. Furthermore, we also give examples of the bounds.

\section{Introduction}
\label{intro}

The Selberg class $S$, defined by Selberg \protect\cite{selberg2}, consists of functions $\mathcal{L}(s)=\sum_{n=1}^{\infty} \frac{a(n)}{n^s}$ which satisfy the following conditions:
\begin{enumerate}
		\item Ramanujan hypothesis: For any $\epsilon>0$ we have $|a(n)| \ll_{\epsilon} n^{\epsilon}$.
		\item \label{analytic} Analytic continuation: There is an integer $k \ge 0$ such that $(s-1)^k\mathcal{L}(s)$ is an entire function of finite order.
		\item Functional equation: There exists a positive integer $f$ and a real number $Q$ and for integer $j \in [1,f]$ there are positive real numbers $\lambda_j$ and complex 			numbers $\omega$, $\mu_j$,  $d_{\mathcal{L}}=2\sum\limits_{j=1}^{f} \lambda_j$, $\lambda=\prod\limits_{j=1}^{f} \lambda_j^{2\lambda_j}$  where $|\omega|=1$ 			and $\Re(\mu_j) \ge 0$ which satisfy
		\begin{equation*}
				\Lambda_{\mathcal{L}}(s)=\omega \overline{\Lambda_{\mathcal{L}}(1-\bar{s})}
		\end{equation*}
		where
		\begin{equation*}
				\Lambda_{\mathcal{L}}(s)=\mathcal{L}(s)Q^s\prod\limits_{j=1}^{f} \Gamma(\lambda_js+\mu_j).
		\end{equation*}
		\item Euler product: We have
		\begin{displaymath}
				\mathcal{L}(s)=\prod_{p} \mathcal{L}_p(s),
		\end{displaymath}
		where
		\begin{displaymath}
				\mathcal{L}_p(s)= \exp\left(\sum\limits_{l=1}^{\infty} \frac{b(p^l)}{p^{ls}}\right)
		\end{displaymath}
		with some coefficients $b(p^l)$ satisfying $b(p^l) \ll p^{l\theta}$ for some $\theta <\frac{1}{2}$.
\end{enumerate}

We denote $s=\sigma+it$ where $\sigma$ and $t$ are real numbers. We also choose the principal branch of the logarithm $\log {\mathcal{L}(s)}$ on the real axis as $\sigma \to \infty$. For other points we use the analytic continuation of the logarithm.

The zeros $s=-\frac{l+\mu_j}{\lambda_j}$ where $l=0, 1, 2,\ldots$ and $j \in [1,f]$ of the function $\mathcal{L}(s)$ are called trivial zeros. We notice that for all trivial zeros it holds that $\sigma \le 0$. Since we suppose that the function has an Euler product, we know that $\mathcal{L}(s)$ does not have zeros for which $\sigma>1$ and $a(1)=1$. The zeros which lie in the strip $0\le \sigma \le 1$ are called non-trivial zeros. The function may have a trivial and a non-trivial zero at the same point. By \cite{conrey} if $d_\mathcal{L}=0$ then $\mathcal{L}(s)\equiv 1$ and if $d_\mathcal{L}>0$ then $d_\mathcal{L}\ge 1$. Since the number of zeros for $\mathcal{L}(s)\equiv 1$, we concentrate to the cases $d_\mathcal{L}\ge 1$.

Riemann-von Mangoldt formula \cite{mangoldt1, mangoldt2} describes the number of the zeros of the Riemann zeta function inside the certain strip. According to this formula, the number of the zeros $\rho$ for which $0 \le \Im(\rho) \le T$ for a positive real number $T$ is
\begin{equation*}
		\frac{T}{2\pi}\log{\frac{T}{2\pi}}-\frac{T}{2\pi}+O(\log{T}).
\end{equation*} 
A. Selberg \cite{selberg} and A. Fujii \cite{fujii} proved similar formulas for the zeros of the $L$-functions inside the critical strip for height $T$ and $[T,T+H]$ respectively. J. Steuding \cite{steuding2} and L. Smaljovi\'{c} \cite{smaljovic} estimated the number of the zeros of the Selberg class functions. The number of the zeros $\rho$ of the function $\mathcal{L} \in S$, for which  $0 \le \Im(\rho) \le T$ or  $-T \le \Im(\rho) \le 0$, is
\begin{equation*}
		\frac{d_{\mathcal{L}}}{2\pi}T\log{\frac{T}{e}}+\frac{T}{2\pi}\log{(\lambda Q^2)}+O(\log{T}).
\end{equation*}
In this paper we are interested in the non-trivial zeros for which $T_0 < \Im(\rho) \le T$ or $-T \le \Im(\rho) < T_0$ when $T> T_0$ is large enough. We prove that
\begin{equation*}
		\left|\mathcal{N}_{\mathcal{L}}^{\pm}(T_0,T) -\frac{d_\mathcal{L}}{2\pi}T\log{\frac{T}{e}}-\frac{T}{2\pi}\log({\lambda Q^2})\right|
		<c_{\mathcal{L},1}\log{T}+c_{\mathcal{L},2}(T_0)+\frac{c_{\mathcal{L},3}(T_0)}{T},
\end{equation*}
where $\mathcal{N}_{\mathcal{L}}^{\pm}(T_0,T)$ denotes the number of the non-trivial zeros of the function $\mathcal{L} \in S$ for $T_0 < \Im(\rho) \le T$ or $-T \le \Im(\rho) < T_0$ when $T>T_0$ is large enough. The terms $c_{\mathcal{L},j}(T_0)$ are real numbers which depend on the function $\mathcal{L}$ and the number $T_0$ and the real number $c_{\mathcal{L},1}$ depends only on the function $\mathcal{L}$. This formula involves no unknown, undefined constants in the error term.This is proved in Theorem \ref{mainResult}. In 1916 R. J. Backlund \cite{backlund} proved similar formula for the Riemann zeta function. E. Carneiro and R. Finder \cite{carneiro}, \cite{finder} proved an explicit bound for the function $S_1(t,\pi)=\frac{1}{\pi}\int_{\frac{1}{2}}^{\infty}\log{|L(\sigma+it,\pi)|}d\sigma$, where $L(s, \pi)$ is in a subset of the $L$-functions, assuming generalized Riemann hypothesis. There are also explicit upper bounds for the number of the zeros for the Riemann zeta-function assuming Riemann hypothesis \cite{carneiro2} and along the critical line \cite{trudgian} and $L$-functions \cite{carneiro3}. G. {Fran{\c c}a} and A. {LeClair} \cite{franca} proved an exact equation for the $n$th zero of the $L$-functions on the critical line. There are also explicit results for Dirichlet $L$-functions and Dedekind zeta-functions, see for example the papers from K. S. McCurley \cite{mccurley} and T. S. Trudgian \cite{trudgian2}.

In the Section \ref{preliminaries} we prove lemmas which are used in the other sections. We also prove a formula which describes the sum of the real parts of the zeros inside certain strip. The sum depends on the four integrals of the function $\mathcal{L}(s)$. The main result follows from the estimates of the integrals. In the Sections \ref{estimatesI2I1} and \ref{estimatesI3I4} we estimate the integrals. In the Section \ref{sectionMainResult} we combine the results of the Sections \ref{preliminaries}, \ref{estimatesI2I1} and \ref{estimatesI3I4} and prove the main result. We follow the proofs of the article \cite{steuding2} and the chapters $6$ and $7$ of the book \cite{steuding}. In the Section \ref{numericalExample} we give examples of the main result.

\section{Preliminary results}
\label{preliminaries}

In this section we prove some preliminary results which are needed to prove the main theorem. The results are used in the Sections \ref{estimatesI2I1}, \ref{estimatesI3I4} and \ref{sectionMainResult}.

\subsection{Basic theory of the function $\mathcal{L}(s)$}

To shorten our notation we define 
\begin{equation*}
		v=\sum\limits_{j=1}^{f}\lambda_j\log\lambda_j, u=\sum\limits_{j=1}^{f}\left(\bar{\mu}_j-\frac{1}{2}\right)\log\lambda_j \text{ and } \mu=4\sum\limits_{j=1}^{f}				\left(\frac{1}{2}-\mu_j\right).	
\end{equation*} 
By the Ramanujan hypothesis ($\epsilon=1$) there is a constant $a_1$ such that for all $n$ we have $|a(n)|\le a_1n$. Note that $a_1 \ge 1$ since we have assumed that $a(n)=1$. Let $a$ be a real number for which $a>2$ and 
\begin{equation*}
		\sum\limits_{n=2}^{\infty}\frac{a_1}{n^{a}}<\frac{1}{2}
\end{equation*}
and let $b<-3$ be a negative real number which has the following property: 
\begin{equation*}
		\sum\limits_{n=2}^{\infty}\frac{a_1}{n^{-b-1}}<1.
\end{equation*} 
In the chapter $7.1$ of the book \cite{steuding} it is proved that: 
\begin{lemma}
\label{summaIntegraalit}
Let $T_0$ and $T>T_0$ be a positive real numbers. Let number $\rho$ denote the zero of the function $\mathcal{L} \in S$. Then
\begin{equation*}
		\begin{aligned}
				2\pi\sum_{\substack{T_0 < \Im(\rho) \le T \\ \Re(\rho)>b}} \left(\Re(\rho)-b\right) &= \int_{T_0}^{T} \log|\mathcal{L}(b+it)| dt-\int_{T_0}^{T}\log|						\mathcal{L}(a+it)| dt \\
				& \quad-\int_{b}^{a} \arg\mathcal{L}(\sigma+iT_0) d\sigma+\int_{b}^{a} \arg\mathcal{L}(\sigma+iT) d\sigma \\
				&\coloneqq I_1(T_0,T,b)+I_2(T_0,T,a)-I_3(T_0,a,b)+I_3(T,a,b).
		\end{aligned}
\end{equation*}
\end{lemma}

The goal is to estimate the integrals and get the main result by using these estimates. By the functional equation we have 
\begin{equation*}
		\mathcal{L}(s)=\Delta_{\mathcal{L}}(s)\overline{\mathcal{L}(1-\bar{s})},
\end{equation*}
where 
\begin{equation}
\label{deltaMaaritelma}
		\Delta_{\mathcal{L}}(s)=\omega Q^{1-2s}\prod\limits_{j=1}^{f}\frac{\Gamma(\lambda_j(1-s)+\bar{\mu}_j)}{\Gamma(\lambda_js+\mu_j)}.
\end{equation}
We use this formula to estimate the term $\log{|\Delta_{\mathcal{L}}(s)|}$. We need this when we estimate the terms $I_1(T_0,T,b)-I_1(T_0,T,b+1)$, $I_3(T_0,a,b)$ and $I_3(T,a, b)$. To do this we define that $B_n$ is $n$th Bernoulli number and we need the following lemma from T. J. Stieltjes \cite[paragraph 9]{stieltjes}:

\begin{lemma}
\label{Stirling}
		Let $z$ be a complex number such that $|\arg(z)|<\pi$. Then for $N=1,2,\ldots$
		\begin{equation*}
				\log{\Gamma(z)}=z\log{z}-z+\frac{1}{2}\log{\frac{2\pi}{z}}+\sum_{n=1}^{N-1}\frac{B_{2n}}{2n(2n-1)z^{2n-1}}+W_N(z),
		\end{equation*}
		where
		\begin{equation*}
				|W_N(z)| \le \frac{B_{2N}}{2N(2N-1)|z|^{2N-1}}\sec^{2N}\left(\frac{\arg{z}}{2}\right)
		\end{equation*}
		is a holomorphic function.
\end{lemma}

Let
\begin{equation*}
		\begin{aligned}
				 V_j(s) &= \left(-\lambda_js+\lambda_j+\bar{\mu}_j-\frac{1}{2} \right)\log{\Big(1+\frac{\lambda_j+\bar{\mu}_j}{-\lambda_js}\Big)}-(\lambda_j+\bar{\mu}					_j) \\
				& \quad-\left(\lambda_js+\mu_j-\frac{1}{2} \right)\log{\Big(1+\frac{\mu_j}{\lambda_js}\Big)}+\mu_j
				 +W(-\lambda_js)-W(\lambda_js),
		\end{aligned}
\end{equation*}
where $|W(z)| \le \frac{B_{2}}{2|z|}\sec^{2}\left(\frac{\arg{z}}{2}\right)$ is a holomorphic function if $|\arg(z)| < \pi$. Let also $V(s)=\sum\limits_{j=1}^{f} V_j(s)$. Using the previous lemma and simplifying expressions we get

\begin{lemma}
\label{logDelta}
Let $\left|\arg\left(\lambda_j(1-s)+\bar{\mu}_j\right)\right|< \pi$, $|\arg(\lambda_js+\mu_j)|<\pi$ and $t>0$. Then
\begin{equation*}
		\begin{aligned}
				\log|\Delta_{\mathcal{L}}(s)| & =\Big(\frac{1}{2}-\sigma\Big)\Big(d_{\mathcal{L}}\log t+\log(\lambda Q^2)\Big)+d_{\mathcal{L}}\sigma \\
				&\quad+\Re\Bigg(\log{\Big(1-\frac{\sigma i}{t}\Big)}\bigg(d_{\mathcal{L}}\Big(\frac{1}{2}-s\Big)+\frac{\Im(\mu)i}{2}\bigg)+V(s)\Bigg)
		\end{aligned}
\end{equation*}		
\end{lemma}

\begin{proof}
		It is enough to estimate the product of the $\Gamma$-functions of the formula \eqref{deltaMaaritelma}. First we do some basic calculation and then we use Lemma 				\ref{Stirling}. After short computations we have 
		\begin{equation*}
				\log(-\lambda_js)=\log(\lambda_j)-\frac{\pi i}{2}+\log{t}+\log{\Big(1-\frac{\sigma i}{t}\Big)}
		\end{equation*}
		and
		\begin{equation*}
				\log(\lambda_js)=\log(\lambda_j)+\frac{\pi i}{2}+\log{t}+\log{\Big(1-\frac{\sigma i}{t}\Big)}.
		\end{equation*}
		We apply Lemma \ref{Stirling} for $N=1$ and using the previous formulas, we get 
		\begin{equation*}
				\begin{aligned}
						& \log\left(\prod\limits_{j=1}^{f}\frac{\Gamma(\lambda_j(1-s)+\bar{\mu}_j)}{\Gamma(\lambda_js+\mu_j)}\right) \\
						&\quad= \sum\limits_{j=1}^{f}\Bigg(\bigg(\lambda_j-\lambda_js+\bar{\mu}_j-\frac{1}{2}\bigg)\bigg(\log(\lambda_j)-\frac{\pi i}{2}+\log{t}+							\log{\Big(1-\frac{\sigma i}{t}\Big)}\bigg)+\lambda_js\Bigg)\\
						&\quad\quad-\sum\limits_{j=1}^{f}\left(\Big(\lambda_js+\mu_j-\frac{1}{2}\Big)\bigg(\log(\lambda_j)+\frac{\pi i}{2}+\log{t}+\log{\Big(1-    							\frac{\sigma i}{t}\Big)}\bigg)-\lambda_js\right) \\
						& \quad\quad+V(s) \\
						&\quad=\log{t}\bigg(d_{\mathcal{L}}\Big(\frac{1}{2}-s\Big)+\frac{\Im(\mu)i}{2}\bigg)+2\Im(u)i+d_{\mathcal{L}}s+2v\Big(\frac{1}{2}-s\Big)\\
						& \quad\quad-\frac{\pi i}{4}\big(d_{\mathcal{L}}-\Re(\mu)\big)+\log{\Big(1-\frac{\sigma i}{t}\Big)}\bigg(d_{\mathcal{L}}\Big(\frac{1}{2}-s\Big)							+\frac{\Im(\mu)i}{2}\bigg)+V(s).
				\end{aligned}
		\end{equation*}
		The claim follows from the previous computations and the definition \eqref{deltaMaaritelma} of the function $\Delta_\mathcal{L}(s)$.
\end{proof}

\subsection{The estimate of the function $V_j(s)$}

The formula of the function $\log|\Delta_{\mathcal{L}}(s)|$ in Lemma \ref{logDelta} contains the term $V(s)$. Thus we want also estimate the term $V(s)$. Since $V(s)=\sum\limits_{j=1}^f V_j(s)$, it is sufficient to estimate the terms $V_j(s)$. Before doing this we prove a lemma. The estimate of the term $V_j(s)$ follows from this lemma.

\begin{lemma}
\label{virheW}
If $t \ge \max\limits_j\big\{\frac{2|\lambda_j+\bar{\mu}_j|}{\lambda_j} \big\}$, $t \ge \max\limits_j \big\{ \frac{2|\mu_j|}{\lambda_j}\big\}$ and $\sigma$ is a constant then for all $j$
		\begin{equation*}
				\begin{aligned}
						& \left| W(-\lambda_js)\right|+\left|W(\lambda_js) \right | \\
						& \quad< \frac{|B_2|}{2|\lambda_jt|}\left(2+ \sec^2\left(\frac{\arg\Big(\lambda_j\big(-|\sigma|+ \max													\limits_j\big\{\frac{2|\lambda_j+\bar{\mu}_j|}{\lambda_j}, \frac{2|\mu_j|}{\lambda_j}\big\}i\big)\Big)}{2}\right) \right).
				\end{aligned}
		\end{equation*}
\end{lemma}

\begin{proof}
		We assume that $\sigma <0$ since we can prove the case $\sigma \ge 0$ similarly. By the definition of the function $W(z)$ we have 
		\begin{equation*}
				|W(z)| \le  \frac{B_{2}}{2|z|}\sec^{2}\left(\frac{\arg{z}}{2}\right)
		\end{equation*} 
		if $|\arg(z)| < \pi$. Also, by the assumptions for the number $t$ we have $|\arg(\pm \lambda_js)|<\pi$. The goal is to estimate the functions $\sec^2\Big(\frac{\arg(-       			\lambda_js)}{2}\Big)$ and $\sec^2\Big(\frac{\arg(\lambda_js)}{2}\Big)$. We do the estimate by using basic properties of the secant function.

		First we estimate the arguments of the complex numbers $-\lambda_js$ and $\lambda_js$. Since $t \ge \max\limits_j\big\{\frac{2|\lambda_j+\bar{\mu}_j|}{\lambda_j} 			\big\}$, $t \ge \max\limits_j \big\{ \frac{2|\mu_j|}{\lambda_j}\big\}$ and $\sigma<0$, the argument of the complex numbers $-\lambda_js$ is in the interval $(-\frac{\pi}			{2},0)$. Thus $\frac{\arg(-\lambda_js)}{2} \in \left(-\frac{\pi}{4},0\right)$. Similarly we have
		\begin{equation*}
				\frac{\arg(\lambda_js)}{2}  \in \left(\frac{\pi}{4},\frac{\arg\Big(\lambda_j\big(\sigma+ \max\limits_j\big\{\frac{2|\lambda_j+\bar{\mu}_j|}{\lambda_j}, 						\frac{2|\mu_j|}{\lambda_j}\big\}i\big)\Big)}{2}\right] \subset \Big(\frac{\pi}{4}, \frac{\pi}{2}\Big).
		\end{equation*}
		Next we use the estimates of the arguments. Because the function $\sec^2(z)$ is an increasing function for $z \in \left[0, \frac{\pi}{2}\right)$ and an even function we 			have
		\begin{equation*}
				\sec^2\Big(\frac{\arg(-\lambda_js)}{2}\Big) < \sec^2\Big(-\frac{\pi}{4}\Big)= 2
		\end{equation*}
		and
		\begin{equation*}
				\sec^2\Big(\frac{\arg(\lambda_js)}{2}\Big) \le \sec^2\left(\frac{\arg\Big(\lambda_j\big(\sigma+ \max\limits_j\big\{\frac{2|\lambda_j+\bar{\mu}_j|}						{\lambda_j}, \frac{2|\mu_j|}{\lambda_j}\big\}i\big)\Big)}{2}\right). 
		\end{equation*}
		The claim follows from the previous equations.
\end{proof}

Now we estimate the term $V_j(s)$.

\begin{lemma}
\label{arvioVj}
If $\sigma$ is a constant and $t$ is as in Lemma \ref{virheW}, then for all $j$
		\begin{equation*}
				\begin{aligned}
						& |V_j(s)| \\
						&  \quad< \left|\frac{(\lambda_j+\bar{\mu}_j)^2}{\lambda_jt} \right|+2\left| \frac{(\lambda_j+\bar{\mu}_j)												(\lambda_j+\bar{\mu}_j-\frac{1}{2})}{\lambda_jt}\right|+ \left|\frac{\mu_j^2}{\lambda_jt} \right|+2\left| \frac{\mu_j(\mu_j-									\frac{1}{2})}{\lambda_jt}\right|\\
						& \quad \quad +\frac{|B_2|}{2|\lambda_js|}\left(2+ \sec^2\left(\frac{\arg\Big(\lambda_j\big(-|\sigma|+ \max\limits_j\big\{\frac{2|\lambda_j+							\bar{\mu}_j|}{\lambda_j}, \frac{2|\mu_j|}{\lambda_j}\big\}i\big)\Big)}{2}\right)\right).
				\end{aligned}
		\end{equation*}
\end{lemma}

\begin{proof}
		By the definition of the function $V_j(s)$ and the series expansion we have
		\begin{equation*}
				\begin{aligned}
						& V_j(s) \\
						& \quad=\left(-\lambda_js+\lambda_j+\bar{\mu}_j-\frac{1}{2} \right)\log{\Big(1+\frac{\lambda_j+\bar{\mu}_j}{-\lambda_js}\Big)}-       								(\lambda_j+\bar{\mu}_j) \\
						&\quad\quad -\left(\lambda_js+\mu_j-\frac{1}{2} \right)\log{\Big(1+\frac{\mu_j}{\lambda_js}\Big)}+\mu_j+W(-\lambda_js)-W(\lambda_js) \\
						& \quad = \frac{(\lambda_j+\bar{\mu}_j)^2}{-\lambda_js}\sum\limits_{n=2}^{\infty} \frac{(-1)^{n+1}}{n}\Big(\frac{\lambda_j+\bar{\mu}_j}{-      						\lambda_js}\Big)^{n-2} \\
						& \quad\quad +\left(\lambda_j+\bar{\mu}_j-\frac{1}{2} \right) \sum\limits_{n=1}^{\infty}\frac{(-1)^{n+1}}{n}\Big(\frac{\lambda_j+								\bar{\mu}_j}{-\lambda_js}\Big)^n-\frac{\mu_j^2}{\lambda_js}\sum\limits_{n=2}^{\infty} \frac{(-1)^{n+1}}{n}\Big(\frac{\mu_j}{\lambda_js}							\Big)^{n-2} \\
						&\quad\quad-\left(\mu_j-\frac{1}{2} \right) \sum\limits_{n=1}^{\infty}\frac{(-1)^{n+1}}{n}\Big(\frac{\mu_j}{\lambda_js}\Big)^n+W(-								\lambda_js)-W(\lambda_js).
				\end{aligned}
		\end{equation*}
		Since $t \ge \max\limits_j\big\{\frac{2|\lambda_j+\bar{\mu}_j|}{\lambda_j} \big\}$ and $t \ge \max\limits_j \big\{ \frac{2|\mu_j|}{\lambda_j}\big\},$ we have $\left|				\frac{\lambda_j+\bar{\mu}_j}{-\lambda_js}\right | \le\frac{1}{2}$ and $\left|\frac{\mu_j}{\lambda_js}\right| \le \frac{1}{2}$ for all $j$. The claim follows from these 				estimates and Lemma \ref{virheW}.
\end{proof}

\section{The difference $I_1(T_0,T,b)-I_1(T_0,T,b+1)$}
\label{estimatesI2I1}

In this section we estimate the integral
\begin{equation*}
		I_1(T_0,T,b)-I_1(T_0,T,b+1)=\int_{T_0}^{T} \log|\mathcal{L}(b+it)|-\log|\mathcal{L}(b+1+it)| dt.
\end{equation*} 

\subsection{Preliminaries for the difference $I_1(T_0,T,b)-I_1(T_0,T,b+1)$}
\label{esitietojaI1}

In this section we prove preliminary results which are used to estimate the term $I_1(T_0,T,b)-I_1(T_0,T,b+1)$. The first one of these describes the properties of the logarithm.

\begin{lemma}
\label{log7}
Let $z$ be a complex number. Then
		\begin{enumerate}[label=(\alph*)]
				\item if $|z| <\frac{1}{2}$  we have
				 \begin{equation*}
						|\log(1+z)|<2|z|,
				\end{equation*} 
				\item if $z$ is a real number we have
				\begin{equation*}
						|\log(1-zi)|<7|z|.
				\end{equation*}
		\end{enumerate}
\end{lemma}

\begin{proof}
		\begin{enumerate}[label=(\alph*)]
				\item Assume that $|z|<\frac{1}{2}$. By the series expansion of the logarithm we have 
				\begin{equation*}
						|\log(1+z)|=\left|\sum\limits_{n=1}^{\infty} \frac{(-1)^{n+1}z^n}{n} \right|\le\sum\limits_{n=1}^{\infty} \left|\frac{z^n}{n}\right|.
				\end{equation*}
				Since $|z|<\frac{1}{2}$ and $n \ge 1$ in the sum, we have
				\begin{equation*}
						\sum\limits_{n=1}^{\infty} \left|\frac{z^n}{n}\right| < \sum\limits_{n=1}^{\infty} |z|\frac{1}{2^{n-1}}=2|z|.
				\end{equation*}
				\item Assume that $z$ is a real number. If $|z|<\frac{1}{2}$ then by the previous case we have $|\log(1-zi)|<2|z|$. Let $|z| \ge \frac{1}{2}$. Now we have
				\begin{equation*}
						\left|\log(1-zi)\right|^2=\left(\log{|1-zi|}\right)^2+\left(\arg(1-zi)\right)^2<5|z|^2+4\pi^2|z|^2.
				\end{equation*}
				Thus
				\begin{equation*}
						|\log(1-zi)| < \sqrt{5+4\pi^2}|z|<7|z|.
				\end{equation*}
		\end{enumerate}
\end{proof}

We use the previous lemma and basic properties of the absolute value to obtain the following two inequalities.

\begin{lemma}
\label{I1apulause1}
If $\sigma<-3$ and $t>0$ then 
		\begin{align*}
				& \left|\Re\Bigg(\log\Big(1-\frac{\sigma i}{t}\Big)\bigg(d_{\mathcal{L}}\Big(\frac{1}{2}-\sigma\Big)+\frac{\Im(\mu)i}{2}\bigg) \right. \\
				&\quad \left.-\log{\Big(1-\frac{(\sigma+1) i}{t}\Big)}\bigg(d_{\mathcal{L}}\Big(-\frac{1}{2}-\sigma\Big)+\frac{\Im(\mu)i}{2}\bigg)\Bigg)\vphantom{\Re}\right| 				\\
				&\quad < \frac{2}{t}\left|-d_{\mathcal{L}}\sigma+\frac{\Im(\mu)i}{2} \right|-7\frac{d_{\mathcal{L}}}{2t}(2\sigma+1).
		\end{align*}
\end{lemma}

\begin{proof}
		We have
		\begin{equation}
		\label{summaLog}
				\begin{aligned}
						& \left|\Re\Bigg(\log\Big(1-\frac{\sigma i}{t}\Big)\bigg(d_{\mathcal{L}}\Big(\frac{1}{2}-\sigma\Big)+\frac{\Im(\mu)i}{2}\bigg) \right. \\
						& \quad\left.-\log{\Big(1-\frac{(\sigma+1) i}{t}\Big)}\bigg(d_{\mathcal{L}}\Big(-\frac{1}{2}-\sigma\Big)+\frac{\Im(\mu)i}{2}\bigg)\Bigg)\right| \\
						&\quad \le \left |\Big(-d_{\mathcal{L}}\sigma+\frac{\Im(\mu)i}{2}\Big)\log\Big(1+\frac{ i}{t-(\sigma+1)i}\Big) \right| \\
						&\quad\quad + \left |\log\Big(1-\frac{\sigma i}{t}\Big)\frac{d_{\mathcal{L}}}{2} \right|+\left |\log\Big(1-\frac{(\sigma+1) i}{t}\Big)									\frac{d_{\mathcal{L}}}{2} \right|.
				\end{aligned}
		\end{equation}
		We want to estimate the previous terms. First we estimate the factor $\log\Big(1+\frac{ i}{t-(\sigma+1)i}\Big)$. Since $\sigma<-3$, we have 
		\begin{equation*}
				\left |\frac{ i}{t-(\sigma+1)i} \right |^2=\frac{1}{t^2+(\sigma+1)^2}<\frac{1}{4}.
		\end{equation*} 
		Thus $|\frac{ i}{t-(\sigma+1)i}|<\frac{1}{2}$ and by Lemma \ref{log7} we get
		\begin{equation}
		\label{logarvio2}
				\begin{aligned}
						& \left|\Big(-d_{\mathcal{L}}\sigma+\frac{\Im(\mu)i}{2}\Big)\log\Big(1+\frac{ i}{t-(\sigma+1)i}\Big)\right| \\
						&\quad< 2 \left|-d_{\mathcal{L}}\sigma+\frac{\Im(\mu)i}{2}\right|\frac{ 1}{|t-(\sigma+1)i|}.
				\end{aligned}
		\end{equation}
		Further, by Lemma \ref{log7} we have
		\begin{equation}
		\label{logarvio71}
				\left |\log\Big(1-\frac{\sigma i}{t}\Big)\frac{d_{\mathcal{L}}}{2} \right | < -7\frac{d_{\mathcal{L}}\sigma}{2t}
		\end{equation}
		and
		\begin{equation}
		\label{logarvio72}
				\left |\log\Big(1-\frac{(\sigma+1) i}{t}\Big)\frac{d_{\mathcal{L}}}{2} \right| < -7\frac{d_{\mathcal{L}}(\sigma+1)}{2t}.
		\end{equation}
		The claim follows from the formulas \eqref{summaLog}, \eqref{logarvio2}, \eqref{logarvio71} and \eqref{logarvio72}.
\end{proof}

\begin{lemma}
\label{I1apulause2}
If $|\sigma| \ge 1$ and $t>0$ then
		\begin{align*}
				& \left|\Re\left(-d_{\mathcal{L}}-d_{\mathcal{L}}\log\Big(1-\frac{\sigma i}{t}\Big)it+d_{\mathcal{L}}\log\Big(1-\frac{(\sigma+1) i}{t}\Big)it\right) \right| \\
				& \quad< d_{\mathcal{L}}\left(\frac{3(\sigma^2+\sigma)}{t^2}+\frac{2}{t}\right).
		\end{align*}
\end{lemma}

\begin{proof}
		We can calculate
		\begin{equation*}
		\label{toisessaMuodossa}
				\begin{aligned}
						& \Re\left(-d_{\mathcal{L}}-d_{\mathcal{L}}\log\Big(1-\frac{\sigma i}{t}\Big)it+d_{\mathcal{L}}\log\Big(1-\frac{(\sigma+1) i}{t}\Big)it\right) \\
						& \quad= -d_{\mathcal{L}}\left(1+t\arg\Big(1-\frac{i}{t-\sigma i}\Big)\right).
				\end{aligned}
		\end{equation*}
		To obtain the absolute value of the term $-d_{\mathcal{L}}\Big(1+t\arg(1-\frac{i}{t-\sigma i})\Big)$ we calculate the upper and lower bounds of this expression. We have 			$\arg(1-\frac{i}{t-\sigma i})=\arctan(\frac{-t}{t^2+\sigma^2+\sigma})$. Thus $\arg(1-\frac{i}{t-\sigma i})\le 0$ and
		\begin{equation*}
		\label{ylaraja}
				-d_{\mathcal{L}}\left(1+t\arg\Big(1-\frac{i}{t-\sigma i}\Big) \right)<-d_{\mathcal{L}}\left( 1+\frac{-t^2}{t^2+\sigma^2+\sigma}\right) \le -d_{\mathcal{L}}					\left(1-1 \right)=0.
		\end{equation*}
		Next we compute the lower bound of the term  $-d_{\mathcal{L}}\Big(1+t\arg(1-\frac{i}{t-\sigma i})\Big)$.

		By \cite{shafer} we have $\arctan(x)\le\frac{3x}{1+2\sqrt{1+x^2}}$ for $x\le 0$. Thus
		\begin{equation}
		\label{arctanArvio}
				\begin{aligned}
						& 1+t\arctan\Big(\frac{-t}{t^2+\sigma^2+\sigma}\Big) \\
						& \quad\le 1-\frac{3t^2}{\left(t^2+\sigma^2+\sigma\right)\left(1+2\sqrt{1+\left(\frac{-t}{t^2+\sigma^2+\sigma}\right)^2}\right)}.
				\end{aligned}
		\end{equation}
		The right hand side is
		\begin{equation}
		\label{murtolukuMuodossa}
				\begin{aligned}
						& = \frac{t^2+\sigma^2+\sigma+2\sqrt{(t^2+\sigma^2+\sigma)^2+t^2}-3t^2}{t^2+\sigma^2+\sigma+2\sqrt{(t^2+\sigma^2+									\sigma)^2+t^2}}.
				\end{aligned}
		\end{equation}
		Since 
		\begin{displaymath}
				t^2+\sigma^2+\sigma+2\sqrt{(t^2+\sigma^2+\sigma)^2+t^2}-3t^2 < 3(\sigma^2+\sigma)+2t
		\end{displaymath}
		and
		\begin{displaymath}
				t^2+\sigma^2+\sigma+2\sqrt{(t^2+\sigma^2+\sigma)^2+t^2} > t^2,
		\end{displaymath}
		by \eqref{arctanArvio} and \eqref{murtolukuMuodossa} we have
		\begin{equation*}
				1+t\arctan\Big(\frac{-t}{t^2+\sigma^2+\sigma}\Big) < \frac{3(\sigma^2+\sigma)}{t^2}+\frac{2}{t},
		\end{equation*}
		as required.
		
\end{proof}

Next we estimate two integrals. The estimates are used in the next section.
\begin{lause}
\label{I2}
Let $T_0$ and $T>T_0$ be positive real numbers. Then 
\begin{equation*}
		\left|\int_{T_0}^{T} \log{|\mathcal{L}(1-b+it)|}\right |<\frac{\pi^2}{3\log 2} \text{ and } \left|\int_{T_0}^{T} \log{|\mathcal{L}(-b+it)|}\right |<\frac{\pi^2}{3\log 2}.
\end{equation*}

\end{lause}

\begin{proof}
		We prove the claim only for the integral $|\int_{T_0}^{T} \log{|\mathcal{L}(-b+it)|}$ since the other case can be proved similarly. First we look at the sum $\mathcal{L}			(-b+it)-1= \sum\limits_{n=2}^{\infty}\frac{a(n)}{n^{-b+it}}$. By the Ramanujan hypothesis and the assumptions for the numbers $a_1$ and $b$
		\begin{displaymath}
				\left | \sum\limits_{n=2}^{\infty}\frac{a(n)}{n^{-b+it}} \right | \le \sum\limits_{n=2}^{\infty}\frac{a_1}{n^{-b-1}}<1.
		\end{displaymath}
		Thus we can use the Taylor series expansion of the logarithm of $\mathcal{L}(-b+it)=1+\sum\limits_{n=2}^{\infty}\frac{a(n)}{n^{-b+it}}$ and get
		\begin{displaymath}
				\log |\mathcal{L}(-b+it)|
				= \Re \left(\sum\limits_{l=1}^{\infty} \frac{(-1)^l}{l} \sum\limits_{n_1=2}^{\infty}\ldots \sum\limits_{n_l=2}^{\infty} \frac{a(n_1)\cdots a(n_l)}                                        				{(n_1\cdots n_l)^{-b+it}}\right).
		\end{displaymath}
		
		We have 
		\begin{align*}
				& \left|\int_{T_0}^{T} \log{|\mathcal{L}(-b+it)|}\right |  \\
				&\quad= \left |\Re \left(\sum\limits_{l=1}^{\infty} \frac{(-1)^l}{l} \sum\limits_{n_1=2}^{\infty}\ldots \sum\limits_{n_l=2}^{\infty} \frac{a(n_1)\cdots a(n_l)}                          				{(n_1\cdots n_l)^{-b}} \int_{T_0}^{T} \frac{dt}{(n_1\cdots n_l)^{it}}\right) \right| \\
		\end{align*}
		
		By the Ramanujan hypothesis
		\begin{align*}
				& \left |\Re \left(\sum\limits_{l=1}^{\infty} \frac{(-1)^l}{l} \sum\limits_{n_1=2}^{\infty}\ldots \sum\limits_{n_l=2}^{\infty} \frac{a(n_1)\cdots a(n_l)}              				{(n_1\cdots n_l)^{-b}} \int_{T_0}^{T} \frac{dt}{(n_1\cdots n_l)^{it}}\right) \right | \\
				&\quad \le \sum\limits_{l=1}^{\infty} \frac{1}{l} \sum\limits_{n_1=2}^{\infty}\ldots \sum\limits_{n_l=2}^{\infty} \frac{a_1^l}{(n_1\cdots                                                                       	  			n_l)^{-b-1}} \left |\int_{T_0}^{T} \frac{dt}{(n_1\cdots n_l)^{it}} \right|.
		\end{align*}
		For $n \ge 2^l$ we have
		\begin{displaymath}
				 \left |\int_{T_0}^{T} \frac{dt}{n^{it}} \right|=\left|\frac{i}{\log n}(e^{-iT\log n}-e^{-iT_0\log n})\right| \le \frac{2}{l\log 2}. 
		\end{displaymath}
		Thus
		\begin{align*}
				& \sum\limits_{l=1}^{\infty} \frac{1}{l} \sum\limits_{n_1=2}^{\infty}\ldots \sum\limits_{n_l=2}^{\infty} \frac{a_1^l}{(n_1\cdots                                                                       	  			n_l)^{-b-1}} \left |\int_{T_0}^{T} \frac{dt}{(n_1\cdots n_l)^{it}}\right | \\
				& \quad \le \sum\limits_{l=1}^{\infty} \frac{2}{l^2 \log 2} \left( \sum\limits_{n=2}^{\infty} \frac{a_1}{n^{-b-1}}\right)^l.
		\end{align*}
		Since we assume that $\sum\limits_{n=2}^{\infty}\frac{a_1}{n^{-b-1}}<1$, we have
		\begin{displaymath}
				\sum\limits_{l=1}^{\infty} \frac{2}{l^2 \log 2} \left( \sum\limits_{n=2}^{\infty} \frac{a_1}{n^{-b-1}}\right)^l <\sum\limits_{l=1}^{\infty} \frac{2}{l^2 						\log 2}=\frac{\pi^2}{3\log 2}.
		\end{displaymath}
\end{proof}

\subsection{The estimate of the difference $I_1(T_0,T,b)-I_1(T_0,T,b+1)$}
\label{sectionI1}

In this section we estimate the term 
\begin{equation*}
		\left|I_1(T_0,T,b)-I_1(T_0,T,b+1)-\int_{T_0}^{T} d_{\mathcal{L}}\log t+\log(\lambda Q^2) dt \right|
\end{equation*} 
which  is used to get the main result. We obtain the estimate by using the results which we have obtained in Sections \ref{preliminaries} and \ref{esitietojaI1}. We want to simplify the notation and thus we define for real numbers $T_0$ and $T$ 
\begin{equation*}
		\begin{aligned}
				& S_j(T_0,T,b)\\
				&\quad = \frac{2}{\lambda_j}\log{\frac{T}{T_0}}\left(\vphantom{\sec^2\left(\frac{\arg\Big(\lambda_j\big(-b+ \max\limits_j\big\{\frac{2|\lambda_j+						\bar{\mu}_j|}{\lambda_j},\frac{2|\mu_j|}{\lambda_j}\big\}i\big)\Big)}{2}\right)}\left|\lambda_j+\bar{\mu}_j \right|^2+2\left|(\lambda_j+\bar{\mu}_j)						\left(\lambda_j+\bar{\mu}_j-\frac{1}{2}\right)\right|+\left|\mu_j\right|^2+2\left|\mu_j\left(\mu_j-\frac{1}{2}\right)\right| 										\right.\\			
				&\quad\quad+\left.\frac{|B_2|}{4}\left(2+ \sec^2\left(\frac{\arg\Big(\lambda_j\big(-b+ \max\limits_j\big\{\frac{2|\lambda_j+\bar{\mu}_j|}							{\lambda_j},\frac{2|\mu_j|}{\lambda_j}\big\}i\big)\Big)}{2}\right) \right)\right.\\
				& \quad\quad\left.+\frac{|B_2|}{4}\left(2+ \sec^2\left(\frac{\arg																		\Big(\lambda_j\big(-b-1+ \max\limits_j\big\{\frac{2|\lambda_j+\bar{\mu}_j|}{\lambda_j}, \frac{2|\mu_j|}{\lambda_j}\big\}i\big)\Big)}{2}\right) \right) \right),	
		\end{aligned}
\end{equation*} 
$S(T_0,T,b)=\sum\limits_{j=1}^{f} S_j(T_0,T,b)$ and
\begin{equation}
\label{R1def}
		\begin{aligned}
				R_1(T_0,T,b)  &  =\log{\frac{T}{T_0}} \left(-7\frac{d_{\mathcal{L}}}{2}(2b+1)+2\left|-d_{\mathcal{L}}b+\frac{\Im(\mu)i}{2} \right| +2d_{\mathcal{L}}					\right)\\
				&\quad+\frac{3d_{\mathcal{L}}(b^2+b)}{T_0}+S(T_0,T,b).
		\end{aligned}
\end{equation}

\begin{lause}
\label{I1}
If $T_0 \ge \max\limits_j \big\{ \frac{2|\lambda_j+\bar{\mu}_j|}{\lambda_j}\big\}$, $T_0 > \max\limits_j \big\{ \frac{2|\mu_j|}{\lambda_j}\big\}$ and $T>T_0$ then
		\begin{align*}
				& \left|I_1(T_0,T,b)-I_1(T_0,T,b+1)-\int_{T_0}^{T} d_{\mathcal{L}}\log t+\log(\lambda Q^2) dt \right| \\
				&\quad < R_1(T_0,T,b)+\frac{2\pi^2}{3\log 2}.
		\end{align*}
\end{lause}

\begin{proof}
		By the definition of the integral $I_1(T_0,T,b)$ and since $\mathcal{L}(s)=\Delta_{\mathcal{L}}(s)\overline{\mathcal{L}(1-\bar{s})}$, we have
		\begin{equation}
		\label{I1kokonaisuudessa}
				\begin{aligned}
						& I_1(T_0,T,b)-I_1(T_0,T,b+1) \\
						& \quad= \int_{T_0}^{T} \log|\Delta_{\mathcal{L}}(b+it)|-\log|\Delta_{\mathcal{L}}(b+1+it)| dt \\ 
						& \quad\quad+\int_{T_0}^{T}\log|\mathcal{L}(1-b+it)|-\log|\mathcal{L}(-b+it)|  dt.
				\end{aligned}
		\end{equation}
		We will first estimate the terms and the claim follows when we sum the estimates.
		
		First we look at the part $\log|\Delta_{\mathcal{L}}(b+it)|-\log|\Delta_{\mathcal{L}}(b+1+it)|$ and its integral. Since $T _0> \max\limits_j \big\{ \frac{2|\mu_j|}				{\lambda_j}\big\}$, for $t\ge T_0$ it holds that
		\begin{equation*}
				\begin{cases}
						\left|\arg\left(\lambda_j(1-b-it)+\bar{\mu}_j\right)\right|< \pi, \\
						|\arg(\lambda_j(b+it)+\mu_j)|<\pi, \\
				 	 	\left|\arg\left(\lambda_j(-b-it)+\bar{\mu}_j\right)\right|< \pi \text{ and } \\
						|\arg(\lambda_j(b+1+it)+\mu_j)|<\pi
				\end{cases} .
		\end{equation*}		
 		Thus by Lemma \ref{logDelta} we have
		\begin{equation*}
				\begin{aligned}
						& \log|\Delta_{\mathcal{L}}(b+it)|-\log|\Delta_{\mathcal{L}}(b+1+it)| \\
						& \quad=\left(d_{\mathcal{L}}\log t+\log(\lambda Q^2)\right)-d_{\mathcal{L}}+\Re\big(V(b+it)-V(b+1+it)\big) \\
						&\quad\quad+ \Re\Bigg(\log\Big(1-\frac{b i}{t}\Big)\bigg(d_{\mathcal{L}}\Big(\frac{1}{2}-b-it\Big)+\frac{\Im(\mu)i}{2}\bigg) \\
						& \quad\quad-\log\Big(1-\frac{(b+1) i}{t}\Big)\bigg(d_{\mathcal{L}}\Big(-\frac{1}{2}-b-it\Big)+\frac{\Im(\mu)i}{2}\bigg)\Bigg).
				\end{aligned}
		\end{equation*}
		Now we estimate the last difference between the last term and the term $d_\mathcal{L}$. By Lemma \ref{I1apulause1} and Lemma \ref{I1apulause2} we have
		\begin{equation*}
				\begin{aligned}
						& \Bigg| \Re\Bigg(\log\Big(1-\frac{b i}{t}\Big)\bigg(d_{\mathcal{L}}\Big(\frac{1}{2}-b-it\Big)+\frac{\Im(\mu)i}{2}\bigg) \\
						& \quad-\log\Big(1-\frac{(b+1) i}{t}\Big)\bigg(d_{\mathcal{L}}\Big(-\frac{1}{2}-b-it\Big)+\frac{\Im(\mu)i}{2}\bigg)-d_{\mathcal{L}}\Bigg) \Bigg| 							\\
						& \quad< \frac{2}{t} \left|-d_{\mathcal{L}}b+\frac{\Im(\mu)i}{2} \right|-7\frac{d_{\mathcal{L}}}{2t}(2b+1)+d_{\mathcal{L}}									\left(\frac{3(b^2+b)}{t^2}+\frac{2}{t}\right).
				\end{aligned}
		\end{equation*}
		We can integrate this and get
		\begin{equation}
		\label{deltaOsaLasku}
				\begin{aligned}
						& \int_{T_0}^{T} \frac{2}{t} \left|-d_{\mathcal{L}}b+\frac{\Im(\mu)i}{2} \right|-7\frac{d_{\mathcal{L}}}{2t}(2b+1)+d_{\mathcal{L}}								\left(\frac{3(b^2+b)}{t^2}+\frac{2}{t}\right) dt \\
						& \quad< \log{\frac{T}{T_0}} \left(-7\frac{d_{\mathcal{L}}}{2}(2b+1)+2\left|-d_{\mathcal{L}}b+\frac{\Im(\mu)i}{2} \right|											+2d_{\mathcal{L}}\right)	+\frac{3d_{\mathcal{L}}(b^2+b)}{T_0}.
				\end{aligned}
		\end{equation}
		Further, we can estimate the integral of the term $\Re\big(V(b+it)-V(b+1+it)\big)$. We remember that $V(s)=\sum_{j=1}^f V_j(s)$. By Lemma \ref{arvioVj}
		\begin{equation*}
				\begin{aligned}
						& |V_j(b+it)|+|V_j(b+1+it)|\\
						& \quad< 2\left|\frac{(\lambda_j+\bar{\mu}_j)^2}{\lambda_jt} \right|+4\left| \frac{(\lambda_j+\bar{\mu}_j)												(\lambda_j+\bar{\mu}_j-\frac{1}{2})}{\lambda_jt}\right|+ 2\left|\frac{\mu_j^2}{\lambda_jt} \right|+4\left| \frac{\mu_j(\mu_j-\frac{1}{2})}							{\lambda_jt}\right|\\
						&\quad\quad+\frac{|B_2|}{2|\lambda_jt|}\left(2+ \sec^2\left(\frac{\arg																	\Big(\lambda_j\big(-b+ \max\limits_j\big\{\frac{2|\lambda_j+\bar{\mu}_j|}{\lambda_j}, \frac{2|\mu_j|}{\lambda_j}\big\}i\big)\Big)}{2}								\right)\right) \\
						&\quad\quad+\frac{|B_2|}{2|\lambda_jt|}\left(2+ \sec^2\left(\frac{\arg\Big(\lambda_j\big(-b-1+ \max\limits_j\big\{\frac{2|\lambda_j+								\bar{\mu}_j|}{\lambda_j}, \frac{2|\mu_j|}{\lambda_j}\big\}i\big)\Big)}{2}\right)\right).
				\end{aligned}
		\end{equation*}
		We can sum these terms, integrate and get
		\begin{equation}
		\label{virheterminReaaliosa}
				\left |\int_{T_0}^{T} \Re\Big(V(b+it)-V(b+1+it)\Big) dt \right| < S(T_0, T,b).
		\end{equation}
		By Theorem \ref{I2}
		\begin{equation}
		\label{I1tokaOsa}
				\left |\int_{T_0}^{T}\log|\mathcal{L}(1-b+it)|-\log|\mathcal{L}(-b+it)| dt \right| < \frac{2\pi^2}{3\log 2}.
		\end{equation}
		The claim follows immediately when we sum the estimates \eqref{I1kokonaisuudessa}, \eqref{deltaOsaLasku}, \eqref{virheterminReaaliosa} and \eqref{I1tokaOsa} 				together.
\end{proof}

\section{Integral $I_3(T,a,b)$}
\label{estimatesI3I4}

In this section we estimate the integral $I_3(T,a,b)=\int_{b}^{a} \arg \mathcal{L}(\sigma+iT)d \sigma$.

\subsection{Preliminaries for the integral $I_3(T,a,b)$}

In this section we collect some preliminary results which are needed in the estimates of the integral $I_3(T,a,b)$.  Our first goal is to estimate the function $\mathcal{L}(s)$, and to do this, we apply the Phragm\'en-Lindel\"of principle for a strip. We also need the following lemma.

\begin{lemma}
\label{L2reaali}
		If $t \ge 1$ then 
		\begin{equation*}
				\Re \left(\log\left(1+\frac{2 i}{t}\right)\left(d_{\mathcal{L}}\left(\frac{5}{2}-it\right)+\frac{\Im(\mu)i}{2}\right)\right) < \frac{5\sqrt{5}+4}{2}     						d_{\mathcal{L}}+|\Im(\mu)|.
		\end{equation*}
\end{lemma}

\begin{proof}
		Since $t \ge 1$, we get
		\begin{equation*}
				\begin{aligned}
						&\Re \left(\log\Big(1+\frac{2 i}{t}\Big)\left(d_{\mathcal{L}}\Big(\frac{5}{2}-it\Big)+\frac{\Im(\mu)i}{2}\right)\right) \\
						& \quad= \log{\left|1+\frac{2i}{t}\right|}\frac{5}{2}d_{\mathcal{L}}+\arg\left(1+\frac{2i}{t}\right)d_{\mathcal{L}}t-\arg\left(1+\frac{2i}{t}							\right)	\frac{\Im(\mu)}{2} \\
						&\quad < \left|1+\frac{2i}{t}\right|\frac{5}{2}d_{\mathcal{L}}+\left|\arctan\left(\frac{2}{t}\right ) d_{\mathcal{L}}t \right|+\left|\arctan								\left(\frac{2}	{t}\right)\frac{\Im(\mu)}{2}\right| \\
						& \quad< \frac{5}{2}\sqrt{1^2+2^2}d_{\mathcal{L}}+\frac{2}{t} d_{\mathcal{L}}t +\frac{2}{t}\left|\frac{\Im(\mu)}{2}\right| \\
						&\quad \le \frac{5\sqrt{5}+4}{2}d_{\mathcal{L}}+|\Im(\mu)|.
				\end{aligned}
		\end{equation*}
\end{proof}

We denote $R=a-b$. Note that $R>0$ and $a-2R<0$. Let 
\begin{equation*}
		\begin{split}
				 M_1
				 &= \max \left \{ \vphantom{ \sec^2\left(\frac{\arg\Big(\lambda_j\big(-b+ \max\limits_j\big\{\frac{2|\lambda_j+\bar{\mu}_j|}{\lambda_j}, 								\frac{2|\mu_j|}{\lambda_j}\big\}i\big)\Big)}{2}\right)}3^k\frac{a_1\pi^2}{6},3^k\frac{a_1\pi^2}{6}|\lambda Q^2|^{\frac{5}{2}} \exp\left( \frac{5\sqrt{5}}					{2}d_{\mathcal{L}}+|\Im(\mu)|+ \frac{1}{\left|\lambda_j(T-2R)\right|} \times \right. \right.\\
				&\quad  \left.\left. \times\sum\limits_{j=1}^f\left(\vphantom{ \sec^2\left(\frac{\arg\Big(\lambda_j\big(-2+ \max\limits_j\big\{\frac{2|\lambda_j+\bar{\mu}					_j|}{\lambda_j},\frac{2|\mu_j|}{\lambda_j}\big\}i\big)\Big)}{2}\right)} 
				\left|\lambda_j+\bar{\mu}_j \right|^2+2\left|(\lambda_j+\bar{\mu}_j)\left(\lambda_j+\bar{\mu}_j-\frac{1}{2}\right)\right|+ \left|\mu_j\right|^2+2\left|					\mu_j\left(\mu_j-\frac{1}{2}\right)\right| \right.\right. \right. \\
				&\quad\left. \left. \left.+\frac{|B_2|}{2}\left(2+ \sec^2\left(\frac{\arg\Big(\lambda_j\big(-2+ \max\limits_j\big\{\frac{2|\lambda_j+									\bar{\mu}_j|}{\lambda_j},\frac{2|\mu_j|}{\lambda_j}\big\}i\big)\Big)}{2}\right) \right) \right) \right) \right \}.
		\end{split}
\end{equation*}
Let also $k$ be defined as in the condition \ref{analytic} for Selberg class. By using the Phragm\'en-Lindel\"of principle for a strip we get the following theorem.

\begin{lause}
\label{I3I4lemma}
		Let 
		\begin{equation*}
		T \ge 
				\begin{cases}
						2R+1 \\
						2R+\max\limits_j\big\{\frac{2|\lambda_j+\bar{\mu}_j|}{\lambda_j}\big\} \\
						2R+\frac{1}{2^{\frac{1}{k}}-1}, \text{ if } k>0
				\end{cases},
		\end{equation*}
		$T>2R+\max\limits_j \big\{ \frac{2|\mu_j|}{\lambda_j}\big\}$ and $t \in [T-2R,T+2R]$. Then $|\mathcal{L}(s)|  \le \frac{a_1\pi^2}{6}$, if $\sigma \ge 3$, 
		\begin{equation*}
				\begin{aligned}
						|\mathcal{L}(s)| &\le  t^{d_{\mathcal{L}}(\frac{1}{2}-\sigma)}\frac{a_1\pi^2}{6}|\lambda Q^2|^{\frac{1}{2}-\sigma}\times \\ 
						& \quad\times\exp\left(d_{\mathcal{L}}\sigma+\Re\left(\log\left(1-\frac{\sigma i}{t}\right)\left(d_{\mathcal{L}}\left(\frac{1}{2}-s\right)+								\frac{\Im(\mu)i}{2}\right)+V(s) \right)\right),
				\end{aligned}
		\end{equation*}
		if $\sigma \le -2$ and $|\mathcal{L}(s)| <2^{\frac{5}{2}d_{\mathcal{L}}+1}M_1 t^{\frac{1}{2}d_{\mathcal{L}}(3-\sigma)}$ if $-2\le \sigma \le 3$.
\end{lause}

\begin{proof}
		We prove the claim in the different parts depending on the value of the real number $\sigma$. Assume $\sigma\ge 3$. Then 
		\begin{equation*}
				|\mathcal{L}(s)|=\left|\sum\limits_{n=1}^{\infty} \frac{a(n)}{n^s} \right | \le \sum\limits_{n=1}^{\infty} \frac{a_1}{n^{2}} = \frac{a_1\pi^2}{6}.
		\end{equation*}
		Next we look at the case $\sigma\le -2$. We have $\mathcal{L}(s)=\Delta_{\mathcal{L}}(s)\overline{\mathcal{L}(1-\bar{s})}$ and by the previous case $|					\overline{\mathcal{L}(1-\bar{s})}|\le\frac{a_1\pi^2}{6}$.  Since $t> \max\limits_j \big\{ \frac{2|\mu_j|}{\lambda_j}\big\}$,  it holds that $\left|\arg						\left(\lambda_j(1-s)+\bar{\mu}_j\right)\right|< \pi$, $|\arg(\lambda_js+\mu_j)|<\pi$ and $t>0$. Thus by Lemma \ref{logDelta} for $\sigma \le -2$ we have
		\begin{equation*}
				\begin{aligned}
						|\mathcal{L}(s)| & \le  t^{d_{\mathcal{L}}(\frac{1}{2}-\sigma)}\frac{a_1\pi^2}{6}|\lambda Q^2|^{\frac{1}{2}-\sigma} \times \\
						&\quad \times \exp\left(d_{\mathcal{L}}\sigma+\Re\left(\log\Big(1-\frac{\sigma i}{t}\Big)\left(d_{\mathcal{L}}\Big(\frac{1}{2}-s\Big)+								\frac{\Im(\mu)i}{2}\right)+V(s)\right)\right).
				\end{aligned}
		\end{equation*}
		
		Now we look at the case $-2 \le \sigma \le 3$. First we look at the function $(s-1)^k\mathcal{L}(s)$. Since $(s-1)^k\mathcal{L}(s)$ is an analytic function of the finite order, 		we can apply the Phragm\'en-Lindel\"of principle for a strip \cite[Theorem 5.53]{iwaniec}. By the case $\sigma\ge3$ and since $t>0$, we have
		\begin{equation*}
				|\mathcal{L}(3+it)| \le \frac{a_1\pi^2}{6}=\frac{a_1\pi^2}{6}(1+t)^{0}.
		\end{equation*}
		Further, by the case $\sigma \le -2$ and Lemma \ref{L2reaali} we have
		\begin{equation*}
				\begin{aligned}
						& |\mathcal{L}(-2+it)| \\
						& \quad\le t^{d_{\mathcal{L}}(\frac{1}{2}+2)}\frac{a_1\pi^2}{6}|\lambda Q^2|^{\frac{1}{2}+2}\exp\left(\vphantom{\Re\left(\log\Big(1+							\frac{2 i}{t}\Big)\left(d_{\mathcal{L}}\Big(\frac{1}{2}+2-it\Big)+\frac{\Im(\mu)i}{2}\right)+V(-2+it)													\right)}-2d_{\mathcal{L}} \right.\\								
						&\quad\quad \left. +\Re\left(\log\Big(1+\frac{2 i}{t}\Big)\left(d_{\mathcal{L}}\Big(\frac{1}{2}+2-it\Big)+\frac{\Im(\mu)i}{2}\right)+V(-2+it)							\right)\right) \\
						&\quad < (1+t)^{\frac{5}{2}d_{\mathcal{L}}}\frac{a_1\pi^2}{6}|\lambda Q^2|^{\frac{5}{2}} \exp\left( 												 \frac{5\sqrt{5}}{2}d_{\mathcal{L}}+|\Im(\mu)|+\sup\limits_{t \in [T-2R,T+2R] } |V(-2+it)| \right).
				\end{aligned}
		\end{equation*}
		Also, from Lemma \ref{arvioVj} it follows that
		\begin{equation}
		\label{VVali}
				\begin{aligned}
						& \sup\limits_{t \in [T-2R,T+2R] }|V(-2+it)| \\
						&\quad < \frac{1}{\left|\lambda_j(T-2R)\right|}\sum\limits_{j=1}^f\left(\vphantom{ \sec^2\left(\frac{\arg\Big(\lambda_j\big(-b+ \max\limits_j							\big\{\frac{2|\lambda_j+\bar{\mu}_j|}{\lambda_j}, \frac{2|\mu_j|}{\lambda_j}\big\}i\big)\Big)}{2}\right)}\left|\lambda_j+\bar{\mu}_j \right|^2							 \right. \\
						&\quad\quad \left.+2\left|(\lambda_j+\bar{\mu}_j)\left(\lambda_j+\bar{\mu}_j-\frac{1}{2}\right)\right|+ \left|\mu_j\right|^2+2\left|\mu_j							\left(\mu_j-\frac{1}{2}\right)\right| \right. \\ 
						&\quad\quad\left. +\frac{|B_2|}{2}\left(2+ \sec^2\left(\frac{\arg\Big(\lambda_j\big(-2+ \max\limits_j\big\{\frac{2|\lambda_j+\bar{\mu}_j|}							{\lambda_j}, \frac{2|\mu_j|}{\lambda_j}\big\}i \big)\Big)}{2}\right) \right) \right).
				\end{aligned}
		\end{equation} 
		Let $l(x)=-\frac{5}{x}+\frac{3}{5}$. We estimate the function $(s-1)^k\mathcal{L}(s)$ in two different cases; $k=0$ and $k>0$. 

		First we look at the case $k=0$. This means that $(s-1)^k\mathcal{L}(s)=\mathcal{L}(s)$. Let
		\begin{equation*}
				M_{\sigma_1}=\frac{a_1\pi^2}{6}|\lambda Q^2|^{\frac{5}{2}} \exp\left(\frac{5\sqrt{5}}{2}d_{\mathcal{L}}+|\Im(\mu)|+\sup									\limits_{t \in [T-2R,T+2R] } |V(-2+it)|\right).	
		\end{equation*} 
		and $M_{\sigma_2}=\frac{a_1\pi^2}{6}$. 
		By the Phragm\'en-Lindel\"of principle for a strip for $-2 \le\sigma \le 3$ and the inequality \eqref{VVali} we have
		\begin{equation}
		\label{Lk0}
				|\mathcal{L}(s)| 
				\le M_{\sigma_1}^{l(\sigma)}M_{\sigma_2}^{1-l(\sigma)}(1+t)^{l(\sigma)\frac{5}{2}d_{\mathcal{L}}} 
				< M_1 (1+t)^{\frac{1}{2}d_{\mathcal{L}}(3-\sigma)}.
		\end{equation}

		Now we look at the case $k > 0$. Since $k$ is an integer, we have $k  \ge 1$. By triangle inequality for $\sigma \in [-2,3]$ we have 
		\begin{equation}
		\label{3k}
				|s-1|^k\le \big(|\sigma-1|+t\big)^k \le 3^k\Big(1+\frac{t}{3}\Big)^k<3^k\big(1+t\big)^k.
		\end{equation}
		This inequality is used when we apply  the Phragm\'en-Lindel\"of principle for a strip. Let
		\begin{equation*}
				M_{\sigma_1} 
				=3^k\frac{a_1\pi^2}{6}|\lambda Q^2|^{\frac{5}{2}} \exp\left(\frac{5\sqrt{5}}{2}d_{\mathcal{L}}+|\Im(\mu)|+\sup\limits_{t \in [T-2R,T+2R] } |V(-2+it)| 					\right)
		\end{equation*} 
		and $M_{\sigma_2}=3^k\frac{a_1\pi^2}{6}$. By the Phragm\'en-Lindel\"of principle for a strip in $-2 \le\sigma \le 3$, the inequality \eqref{VVali} and the inequality 				\eqref{3k} we have
		\begin{equation*}
				|(s-1)^k\mathcal{L}(s)|
				\le M_{\sigma_1}^{l(\sigma)}M_{\sigma_2}^{1-l(\sigma)}(1+t)^{l(\sigma)(k+\frac{5}{2}d_{\mathcal{L}})+k(1-l(\sigma))}
				< M_1 (1+t)^{\frac{1}{2}d_{\mathcal{L}}(3-\sigma)+k}.
		\end{equation*}
		Thus $|\mathcal{L}(s)| < M_1 \left(\frac{1+t}{|s-1|}\right)^k (1+t)^{\frac{1}{2}d_{\mathcal{L}}(3-\sigma)}$. Since by assumptions $t \in [T-2R,T+2R]$, we have $t \ge 			\frac{1}{2^{\frac{1}{k}}-1}$. Thus
		\begin{equation}
		\label{Lk1}
				M_1 \left(\frac{1+t}{|s-1|}\right)^k (1+t)^{\frac{1}{2}d_{\mathcal{L}}(3-\sigma)} \le 2M_1(1+t)^{\frac{1}{2}d_{\mathcal{L}}(3-									\sigma)} .
		\end{equation}
		
		By the inequalities \eqref{Lk0} and \eqref{Lk1} we have 
		\begin{equation*}
				|\mathcal{L}(s)|< M_1 (1+t)^{\frac{1}{2}d_{\mathcal{L}}(3-\sigma)}< 2M_1(1+t)^{\frac{1}{2}d_{\mathcal{L}}(3-\sigma)}
		\end{equation*} 
		if $k=0$ and 
		\begin{equation*}
				|\mathcal{L}(s)|< 2M_1(1+t)^{\frac{1}{2}d_{\mathcal{L}}(3-\sigma)}
		\end{equation*}
		if $k>0$. Further, for $\sigma \in [-2,3]$ and $t \ge 1$ we have
		\begin{equation*} 
				 (1+t)^{\frac{1}{2}d_{\mathcal{L}}(3-\sigma)} \le (2t)^{\frac{1}{2}d_{\mathcal{L}}(3-\sigma)} \le 2^{\frac{5}{2}d_{\mathcal{L}}}t^{\frac{1}{2}						d_{\mathcal{L}}(3-\sigma)}.
		\end{equation*}
		Thus
		\begin{equation*}
				|\mathcal{L}(s)| < 2^{\frac{5}{2}d_{\mathcal{L}}+1}M_1 t^{\frac{1}{2}d_{\mathcal{L}}(3-\sigma)}
		\end{equation*}
		for all $k$.
 
\end{proof}

The following property is also useful in estimating the integral $I_3(T,a,b)$.
\begin{lemma}
\label{I3I4lemmalause}
		Let $\sigma \in [a-2R, a+2R]$, $T>2R$ and $t \in [T-2R,T+2R]$. Then
		\begin{align*}
				& \Re\left(\log\Big(1-\frac{\sigma i}{t}\Big)\bigg(d_{\mathcal{L}}\Big(\frac{1}{2}-s\Big)+\frac{\Im(\mu)i}{2}\bigg)\right) \\
				&\quad < \left|1-\frac{(a+2R)i}{T-2R} \right|d_{\mathcal{L}}\Big(\frac{1}{2}-a+2R\Big)+d_{\mathcal{L}}(a+2R)+\frac{a+2R}{T-2R}\left|\frac{\Im(\mu)}					{2}\right|.
		\end{align*}
\end{lemma}

\begin{proof}
		We have
		\begin{equation}
		\label{reaaliOsa}
				\begin{aligned}
						& \Re\left(\log\Big(1-\frac{\sigma i}{t}\Big)\bigg(d_{\mathcal{L}}\Big(\frac{1}{2}-s\Big)+\frac{\Im(\mu)i}{2}\bigg)\right) \\
						& \quad=\log \left|1-\frac{\sigma i}{t}\right|d_{\mathcal{L}}\Big(\frac{1}{2}-\sigma\Big)+\arg\Big(1-\frac{\sigma i}{t}\Big)\left(d_{\mathcal{L}}     						t-\frac{\Im(\mu)}{2}\right).
				\end{aligned}
		\end{equation}
		By the assumptions for the numbers $\sigma$ and $t$ we have $|\sigma| \le a+2R$ and $t \ge T-2R >0$. Thus
		\begin{equation}
		\label{log1sigmaArvio}
				\log \left|1-\frac{\sigma i}{t}\right|d_{\mathcal{L}}\Big(\frac{1}{2}-\sigma\Big) < \left |1-\frac{(a+2R)i}{T-2R} \right|d_{\mathcal{L}}\left(\frac{1}{2}-a+2R					\right)
		\end{equation}
		and 
		\begin{equation}
		\label{arg1sigmaArvio}
				\arg\Big(1-\frac{\sigma i}{t}\Big)\left(d_{\mathcal{L}}t-\frac{\Im(\mu)}{2}\right) < d_{\mathcal{L}}(a+2R)+\frac{a+2R}{T-2R}\left|\frac{\Im(\mu)}{2}						\right|.
		\end{equation}
		The claim follows immediately from the formulas \eqref{reaaliOsa}, \eqref{log1sigmaArvio} and \eqref{arg1sigmaArvio}.
\end{proof}

\subsection{The estimate of the integral $I_3(T,a,b)$}
\label{sectionI3}

In this section we estimate  the integral $I_3(T,a,b)$. Since the estimate contains the term $V(s)$, where $\sigma$ and $t$ lie in specific intervals, we also need estimate the term $V(s)$ on these intervals. We want to shorten our notation and thus we define the following terms: Let $T$ be a real number,
\begin{equation*}
		\begin{aligned}
				& V^*(T) \\
				& \quad = \frac{1}{T-2R}\sum\limits_{j=1}^f\frac{1}{\lambda_j}\left(\vphantom{\sec^2\left(\frac{\arg\Big(\lambda_j\big(-a-2R+\max\limits_j							\big
				\{\frac{2|\lambda_j+\bar{\mu}_j|}{\lambda_j}, \frac{2|\mu_j|}{\lambda_j}\big\}i\big)\Big)}{2}\right)}\left|\lambda_j+\bar{\mu}_j 									\right|^2  \right. \\
				&\quad\quad \left.+2\left|(\lambda_j+\bar{\mu}_j)\left(\lambda_j+\bar{\mu}_j-\frac{1}{2}\right)\right|+ \left|\mu_j \right|^2+2\left| \mu_j\left(\mu_j-						\frac{1}{2}\right)\right| \right. \\
				& \quad\quad \left.+\frac{|B_2|}{2}\left(2+ \sec^2\left(\frac{\arg\Big(\lambda_j\big(-a-2R+\max\limits_j\big\{\frac{2|\lambda_j+\bar{\mu}_j|}{\lambda_j}, 					\frac{2|\mu_j|}{\lambda_j}\big\}i\big)\Big)}{2}\right) \right) \right )
		\end{aligned}
\end{equation*}
and
\begin{equation}
\label{R2def}
		\begin{aligned}
				& R_2(T) \\
				&\quad = \frac{1}{\log 2} \left(\vphantom{ \max\bigg\{0, \frac{5}{2}\log{|\lambda Q^2|}+\frac{5\sqrt{5}}{2}												d_{\mathcal{L}}+|\Im(\mu)| \bigg\} \Bigg\} }d_{\mathcal{L}}\left(\frac{1}{2}-a+2R\right)\log (2T) +\log{\frac{a_1\pi^2}{6}}+V^*(T) \right. \\			
				&\quad\quad \left.+\max \Bigg\{\max\bigg\{\frac{5}{2}\log{|\lambda Q^2|}, \left(\frac{1}{2}-a+2R\right)\log{|\lambda Q^2|} \bigg								\}-2d_{\mathcal{L}} \right.\\
				& \quad\quad\left.+\left|1-\frac{(a+2R)i}{T-2R} \right|d_{\mathcal{L}}\left(\frac{1}{2}-a+2R\right)+d_{\mathcal{L}}(a+2R)+\frac{a+2R}{T-2R}\left|						\frac{\Im(\mu)}{2} \right|, \right.\\
				&\quad\quad \left.  \left(\frac{5}{2}d_{\mathcal{L}}+1\right)\log{2}+k\log{3}+ \max\bigg\{0, \frac{5}{2}\log{|\lambda Q^2|}+\frac{5\sqrt{5}}{2}						d_{\mathcal{L}}+|\Im(\mu)| \bigg\} \Bigg\} \right)
		\end{aligned}
\end{equation}

\begin{lemma}
\label{ylarajapositiivinen}
		Assume that $T$ satisfies the same conditions as in Theorem \ref{I3I4lemma}.
		Let 
		\begin{equation*}
				g(z)=\frac{1}{2}\left(\mathcal{L}(z+iT)+\overline{\mathcal{L}(\bar{z}+iT)}\right).
		\end{equation*}
		Then 
		\begin{equation*}
				\left |\frac{1}{2\pi \log 2}\int_{0}^{2\pi} \log |g(a+2Re^{i \theta})| d\theta \right|< R_2(T)
		\end{equation*} 
		if $\int_{0}^{2\pi} \log |g(a+2Re^{i \theta})| d \theta \ge 0$.
\end{lemma}

\begin{proof}
		We assume that  $\int_{0}^{2\pi} \log |g(a+2Re^{i \theta})| d \theta \ge 0$. First we estimate the term $|g(a+2Re^{i \theta})|$. To do this we 							estimate the functions $\mathcal{L}(z+iT)$ and $\overline{\mathcal{L}(\bar{z}+iT)}$ by Theorem \ref{I3I4lemma} and Lemma \ref{I3I4lemmalause}. First we define the 				function $M(T)$. Let 
		\begin{equation*}
				\begin{aligned}
						M(T) &=\frac{a_1\pi^2}{6}\max \Bigg\{|\lambda Q^2|^\frac{5}{2}\exp\Bigg(-2d_{\mathcal{L}}+\left|1-\frac{(a+2R)i}{T-2R} \right|								d_{\mathcal{L}}\left(\frac{1}{2}-a+2R\right) \\ 
						&\quad+d_{\mathcal{L}}(a+2R)+\frac{a+2R}{T-2R}\left|\frac{\Im(\mu)}{2} \right|\Bigg), |\lambda Q^2|^{\frac{1}{2}-a+2R}\exp									\Bigg(-2d_{\mathcal{L}} \\
						&\quad+\left|1-\frac{(a+2R)i}{T-2R} \right|d_{\mathcal{L}}\left(\frac{1}{2}-a+2R\right)+d_{\mathcal{L}}(a+2R)+\frac{a+2R}{T-2R}\left|							\frac{\Im(\mu)}{2} \right|\Bigg), \\
						&\quad 2^{\frac{5}{2}d_{\mathcal{L}}+1}3^k, 2^{\frac{5}{2}d_{\mathcal{L}}+1}3^k|\lambda Q^2|^\frac{5}{2}\exp\bigg(\frac{5\sqrt{5}}							{2}d_{\mathcal{L}}+|\Im(\mu)|\bigg) \Bigg\}.
				\end{aligned}
		\end{equation*}
		Since the estimates of the functions $\mathcal{L}(z+iT)$ and $\overline{\mathcal{L}(\bar{z}+iT)}$ contain restrictions of the imaginary and real parts, we estimate 				them. We have
		\begin{equation*}
				\Re(a+2Re^{\pm i \theta}+iT)=a+2R\cos(\theta) \in [a-2R,a+2R].
		\end{equation*}
		Further, we have
		\begin{displaymath}
				 |\Im(a+2Re^{\pm i \theta}+iT)|=|2R\sin(\pm\theta)+T| \in[T-2R,T+2R].
		\end{displaymath}
		Note that $[-2,3] \subset [a-2R,a+2R]$. Further, note that by Lemma \ref{arvioVj} 
		\begin{equation*}
				\sup_{\sigma \in [a-2R, a+2R],t \in [T-2R,T+2R] } |V(s)| < V^*(T). 
		\end{equation*}
		Thus by Theorem \ref{I3I4lemma} and Lemma \ref{I3I4lemmalause} 
		\begin{displaymath}
				|\mathcal{L}(a+2Re^{\pm i \theta}+iT)| < (2R+T)^{d_{\mathcal{L}}(\frac{1}{2}-a+2R)}M(T)\exp(V^*(T)).
		\end{displaymath}
		The same estimate holds also for $\overline{\mathcal{L}(\bar{z}+iT)}$. Thus 
		\begin{equation}
		\label{arviog}
				\left|g(a+2Re^{i \theta})\right| < (2R+T)^{d_{\mathcal{L}}(\frac{1}{2}-a+2R)}M(T)\exp(V^*(T)).
		\end{equation}
		
		We have estimated the term $|g(a+2Re^{i \theta})|$. Next we estimate the term $\left |\frac{1}{2\pi \log 2}\int_{0}^{2\pi} \log |g(a+2Re^{i \theta})| 			                     d\theta \right|$. Let
		\begin{equation*}
				\chi_{\log{|g| \ge 0}}=
				\begin{cases}
						1, \textrm{ if } \log{|g| \ge 0} \\
						0, \textrm{ otherwise }
				\end{cases} .
		\end{equation*}			
		Now, since  $\int_{0}^{2\pi} \log |g(a+2Re^{i \theta})| d \theta \ge 0$, we have
		\begin{displaymath}
				\left|\int_{0}^{2\pi} \log |g(a+2Re^{i \theta})| d \theta \right | \le \int_{0}^{2\pi} \chi_{\log{|g| \ge 0}} \log |g(a+2Re^{i \theta})| d \theta.
		\end{displaymath}
		Since for $\log|g(a+2Re^{i \theta})| \ge 0$ it holds that $|g(a+2Re^{i \theta})| \ge 1$, it is enough to know the upper bound of the term $|g(a+2Re^{i \theta})|$. 				Further, by the inequality \eqref{arviog} we have 
		\begin{equation*}
				\begin{aligned}
				& \int_{0}^{2\pi} \chi_{\log{|g| \ge 0}}\log |g(a+2Re^{i \theta})| d \theta \\
				&\quad< \int_{0}^{2\pi} \left|\log \left|(2R+T)^{d_{\mathcal{L}}(\frac{1}{2}-a+2R)}M(T)\exp(V^*(T))\right| \right |d\theta.
				\end{aligned}
		\end{equation*}
		
		The last step is to integrate the terms. By the definition of the functions $M(T)$ and $V^*(T)$ we have
		\begin{align*}
				&  \left|\log \left|(2R+T)^{d_{\mathcal{L}}(\frac{1}{2}-a+2R)}M(T)\exp(V^*(T))\right| \right | \\
				& \quad\le d_{\mathcal{L}}\left(\frac{1}{2}-a+2R\right)\log (2T) +\log{\frac{a_1\pi^2}{6}}+\max \left\{\max\Bigg\{\frac{5}{2}\log{|\lambda Q^2|},						\right.\\
				& \quad\quad\left. \left(\frac{1}{2}-a+2R\right)\log{|\lambda Q^2|}\Bigg\}-2d_{\mathcal{L}} +\left|1-\frac{(a+2R)i}{T-2R} \right|d_{\mathcal{L}}						\left(\frac{1}{2}-a+2R\right) \right. \\
				&\quad\quad \left.+d_{\mathcal{L}}(a+2R)+\frac{a+2R}{T-2R}\left|\frac{\Im(\mu)}{2} \right|, \left(\frac{5}{2}d_{\mathcal{L}}+1\right)\log{2}+k\log{3}					\right. \\ 
				&\quad\quad \left. +\max\Bigg\{0,\frac{5}{2}\log{|\lambda Q^2|}+\frac{5\sqrt{5}}{2}d_{\mathcal{L}}+|\Im(\mu)|\Bigg\} \right\}+V^*(T).  
		\end{align*}
		Thus we have
		\begin{equation*}
				\begin{aligned}
				 		& \left |\frac{1}{2\pi \log 2}\int_{0}^{2\pi} \log |g(a+2Re^{i \theta})| d\theta \right| \\
						&\quad < \frac{1}{\log 2} \left(\vphantom{ \max\bigg\{0, \frac{5}{2}\log{|\lambda Q^2|}+\frac{5\sqrt{5}}{2}												d_{\mathcal{L}}+|\Im(\mu)| \bigg\} \Bigg\} }d_{\mathcal{L}}\left(\frac{1}{2}-a+2R\right)\log (2T) +\log{\frac{a_1\pi^2}{6}}+V^*(T) \right. \\	
						&\quad\quad \left.+\max \Bigg\{\max\bigg\{\frac{5}{2}\log{|\lambda Q^2|}, \left(\frac{1}{2}-a+2R\right)\log{|\lambda Q^2|} \bigg								\}-2d_{\mathcal{L}} \right.\\
						& \quad\quad\left.+\left|1-\frac{(a+2R)i}{T-2R} \right|d_{\mathcal{L}}\left(\frac{1}{2}-a+2R\right)+d_{\mathcal{L}}(a+2R)+\frac{a+2R}							{T-2R}\left|\frac{\Im(\mu)}{2} \right|, \right.\\
						&\quad\quad \left.  \left(\frac{5}{2}d_{\mathcal{L}}+1\right)\log{2}+k\log{3}+ \max\bigg\{0, \frac{5}{2}\log{|\lambda Q^2|}+\frac{5\sqrt{5}}							{2}d_{\mathcal{L}}+|\Im(\mu)| \bigg\} \Bigg\} \right) \\
						&\quad = R_2(T).
				\end{aligned}
		\end{equation*}
\end{proof}

We define a new function $n(r)$ and estimate it. The estimate is used to estimate the integral $I_3(T,a,b)$.

\begin{lause}
\label{I31}
		Assume that $T$ satisfies the same conditions as in Theorem \ref{I3I4lemma}. Let $n(r)$ be the number of the zeros of the function $g(z)$ in $|z-a| \le r$, where $g(z)$ is as 			in Lemma \ref{ylarajapositiivinen}. Then $n(R) < R_2(T)+1$.
\end{lause}
\begin{proof}
		First we estimate the value of the function $n(R)$ with the function $g(z)$ and its integral. Then we estimate the previous terms and obtain $n(R) < R_2(T)+1$.

		Since $\Im(z+iT)>0$ and $\Im(\bar{z}+iT)>0$ for $|z-a|<T$, the functions $\mathcal{L}(z+iT)$ and $\overline{\mathcal{L}(\bar{z}+iT)}$  are analytic in the disc $|z-a|<T			$. Thus the function $g(z)$ is analytic in the disc $|z-a|<T$. We have
		\begin{displaymath}
				\int_{0}^{2R} \frac{n(r)}{r} dr \ge \int_{R}^{2R} \frac{n(R)}{r} dr=n(R)\log(2).
		\end{displaymath}
		By Jensen's formula
		\begin{displaymath}
				\int_{0}^{2R} \frac{n(r)}{r} dr=\frac{1}{2\pi}\int_{0}^{2\pi} \log \left|g(a+2Re^{i \theta})\right| d\theta-\log |g(a)|.
		\end{displaymath}
		Thus 
		\begin{equation}
		\label{galaraja1}
				n(R) \le \frac{1}{2\pi \log 2}\int_{0}^{2\pi} \log \left|g(a+2Re^{i \theta})\right| d\theta-\frac{\log |g(a)|}{\log 2}.
		\end{equation}
		
		From the previous formula and Lemma \ref{ylarajapositiivinen} we see that it is enough to estimate the terms 
		\begin{equation*} 
				\frac{1}{2\pi \log 2}\int_{0}^{2\pi} \log \left|g(a+2Re^{i \theta})\right| d\theta
		\end{equation*} 
		if $\int_{0}^{2\pi} \log |g(a+2Re^{i \theta})| d \theta < 0$ and $\frac{\log |g(a)|}{\log 2}$.  By the assumptions for the number $a$ we have
		\begin{displaymath}
				\Re\big(\mathcal{L}(a+iT)\big)=\Re\left(1+\sum\limits_{n=2}^{\infty} \frac{a(n)}{n^{a+iT}}\right) \in \left[\frac{1}{2}, \frac{3}{2}\right].
		\end{displaymath}
		Since $g(a)=\Re(\mathcal{L}(a+iT))$, we have $\left|\log |g(a)|\right| \le \log 2$. Thus 
		\begin{equation}
		\label{galaraja2}
				\left|\frac{\log |g(a)|}{\log 2}\right|\le 1.
		\end{equation}

		Next we look at the case $\int_{0}^{2\pi} \log |g(a+2Re^{i \theta})| d \theta <0$. By the definition of the function $n(R)$ we have $n(R) \ge 0$. Also by 					\eqref{galaraja1} and \eqref{galaraja2}
		\begin{displaymath}
				-1 \le -\left|\frac{\log |g(a)|}{\log 2}\right|+n(R)  \le \frac{1}{2\pi \log 2}\int_{0}^{2\pi} \log \left|g(a+2Re^{i \theta})\right| d\theta.	
		\end{displaymath}
		Thus 
		\begin{equation}
		\label{ylarajanegatiivinen}
				\left|\frac{1}{2\pi \log 2}\int_{0}^{2\pi} \log \left|g(a+2Re^{i \theta})\right| d\theta \right| \le 1
		\end{equation}
		if $\int_{0}^{2\pi} \log |g(a+2Re^{i \theta})| d \theta < 0$.
		By \eqref{galaraja1}, \eqref{galaraja2}, \eqref{ylarajanegatiivinen} and Lemma \ref{ylarajapositiivinen} we get the result
		\begin{equation*}
				n(R) < R_2(T)+1.
		\end{equation*}
\end{proof}

In the following theorem we estimate the integral $I_3(T,a,b)=\int_{b}^{a} \mathcal{L}(\sigma+iT)d\sigma$ using the previous theorem.

\begin{lause}
\label{I32}
If $T$ satisfies the same conditions as in Theorem \ref{I3I4lemma}, then
		\begin{equation*}
				|I_3(T,a,b)| < \pi R\left(R_2(T)+2\right).
		\end{equation*}
\end{lause}

\begin{proof}
		Assume that the function $\Re\left(\mathcal{L}(\sigma+iT)\right)$ has $N$ zeros for $b\le \sigma \le a$. Now the sign of the $\Re(\mathcal{L}(\sigma+iT))$ changes at 			most $N+1$ times in the interval $\sigma \in [b,a]$. We can divide the interval $[b,a]$ to $N+1$ parts where $\Re(\mathcal{L}(\sigma+iT))$ is of constant sign. When 			we sum the maximum absolute values of the argument in each of these intervals, we get
		\begin{equation*}
		\label{arvioI3I4N}
				\big|\arg(\mathcal{L}(\sigma+iT))\big|\le \pi(N+1).
		\end{equation*} 
		Let $n(R)$ and $g(z)$ be as in Theorem \ref{I31}. Since $(b,a) \subseteq \{z :|z-a|\le R\}$ and $g(\sigma)=\Re(\mathcal{L}(\sigma+iT))$, we have $N\le n(R)$. Thus
		\begin{equation*}
				\pi(N+1) \le \pi (n(R)+1).
 		\end{equation*}
		By Theorem \ref{I31} we have $ \pi (n(R)+1) < \pi (R_2(T)+2)$. It follows that
		\begin{equation*}
				|I_3(T,a,b)| \le \int_{b}^{a} |\arg\mathcal{L}(\sigma+iT)| d\sigma <  \pi R(R_2(T)+2).
		\end{equation*}
\end{proof}

\begin{remark}
\label{noteI3b1}
		Since $(b+1,a) \subset (b,a)$ it also holds that $|I_3(T,a,b+1)| < \pi R(R_2(T)+2)$.
\end{remark}

\section{Main result}
\label{sectionMainResult}

In this section we prove the explicit version of the Riemann-von Mangoldt type formula for the functions of the set $S$. Let $\mathcal{N}_\mathcal{L}^{+}(T_0,T)$ and $\mathcal{N}_\mathcal{L}^{-}(T_0,T)$ be the number of the non-trivial zeros $\rho$ of the function $\mathcal{L}(s)$ with $T_0 < \Im(\rho) \le T$ and $-T \le \Im(\rho) < -T_0$ respectively. First we combine the results from the Sections \ref{preliminaries}, \ref{estimatesI2I1} and \ref{estimatesI3I4} to estimate the functions $\mathcal{N}_\mathcal{L}^{+}(T_0,T)$ and $\mathcal{N}_\mathcal{L}^{-}(T_0,T)$. We remember that $a_1$ is a constant such that for all $n$ we have $|a(n)|\le a_1n$. Also, $a$ is a real number for which $a>2$ and 
\begin{equation*}
		\sum\limits_{n=2}^{\infty}\frac{a_1}{n^{a}}<\frac{1}{2}
\end{equation*} 
and $b<-3$ is a negative real number which has the following property:  
\begin{equation*}
		\sum\limits_{n=2}^{\infty}\frac{a_1}{n^{-b-1}}<1.
\end{equation*} 
We have defined that $R=a-b$. The constants $d_\mathcal{L}$, $\lambda$, $Q$, $f$, $\mu_j$ and $\lambda_j$ depend on the function $\mathcal{L}$ and are defined at the beginning of the Section \ref{intro}. The function $R_1(T_0,T,b)$ is defined in the formula \eqref{R1def} of the Section \ref{sectionI1} and the function $R_2(t)$ is defined in the formula \eqref{R2def} of the Section \ref{sectionI3}. Let
\begin{equation*}
		\begin{aligned}
				&R_\mathcal{L}(T_0,T)\\
				&\quad= \frac{d_{\mathcal{L}}}{2 \pi}T_0 \log \frac{T_0}{e}+\frac{T_0}{2\pi}\left|\log(\lambda Q^2)\right|+\frac{R_1(T_0,T,b)}{2\pi}+\frac{\pi}						{3\log{2}} \\
				&\quad\quad +\left(R-\frac{1}{2}\right)\left(R_2(T_0)+R_2(T)+4\right)+ f\cdot\left(\left|(b+1)\max_j\{\lambda_j\}+\min_j\{\Re(\mu_j)\}\right| \right. \\
				& \quad\quad \left. \vphantom{\left|(b+1)\max_j\{\lambda_j\}+\min_j\{\Re(\mu_j)\}\right|}-b\max_j\{\lambda_j\}+(b+1)\min_j\{\lambda_j\}-							\min_j\{\Re(\mu_j)\}+\max_j\{\Re(\mu_j)\}\right).
		\end{aligned}
\end{equation*}
 
First we prove a useful lemma and then we use it to prove the main result.
\begin{lemma}
\label{N+}
		Assume that $T_0$ satisfies the same conditions as $T$ in Theorem \ref{I3I4lemma} and $T>T_0$ is a real number. Then 
		\begin{displaymath}
				\left | \mathcal{N}_\mathcal{L}^{\pm}(T_0,T)-\frac{d_\mathcal{L}}{2\pi}T\log{\frac{T}{e}}-\frac{T}{2\pi}\log({\lambda Q^2}) \right  |<R_							\mathcal{L}(T_0,T).
		\end{displaymath}
\end{lemma}

\begin{proof}
		Since the functions $\mathcal{L}(\bar{s})$ and $\overline{\mathcal{L}(\bar{s})}=\sum\limits_{n=1}^\infty \frac{\overline{a(n)}}{n^s}$ have the same zeros, we need to 			prove the claim only for the function $\mathcal{N}_\mathcal{L}^{+}(T_0,T)$. By Lemma \ref{sumIntegrals} for the zeros $\rho$ of the function $\mathcal{L}$ we 				have
		\begin{displaymath}
				2\pi\sum_{\substack{T_0 < \Im(\rho) \le T \\ \Re(\rho) > b}} (\Re(\rho)-b)=I_1(T_0,T,b)-I_2(T_0,T,a)-I_3(T_0,a,b)+I_3(T,a,b).
		\end{displaymath}
		We can subtract the formula containing $b+1$ from the formula containing $b$ and get
		\begin{align*}
				& 2\pi \mathcal{N}_\mathcal{L}^{+}(T_0,T)+2\pi\sum_{\substack{T_0 < \Im(\rho) \le T \\ 0>\Re(\rho)>b+1}} 1+2\pi\sum_{\substack{T_0 < \Im(\rho) \le T 					\\ b+1 \ge \Re(\rho)>b}}(\Re(\rho)-b) \\
				& \quad= I_1(T_0,T,b)-I_1(T_0,T,b+1)-I_3(T_0,a,b) \\
				&\quad\quad+I_3(T_0,a,b+1)+I_3(T,a,b)-I_3(T,a, b+1).
		\end{align*}
		By Theorems \ref{I1}, \ref{I32} and Remark \ref{noteI3b1} we have
		\begin{equation}
		\label{NIntegral}
				\begin{aligned}
						& \left|\mathcal{N}_\mathcal{L}^{+}(T_0,T)-\frac{d_\mathcal{L}}{2\pi}T\log{\frac{T}{e}}-\frac{T}{2\pi}\log({\lambda Q^2})									\right | \\
						& \quad< \frac{d_{\mathcal{L}}}{2 \pi}T_0 \log \frac{T_0}{e}+\frac{T_0}{2\pi}\left|\log(\lambda Q^2)\right|+\frac{R_1(T_0,T,b)}{2\pi}+							\frac{\pi}{3\log{2}} \\
						&\quad\quad +\left(R-\frac{1}{2}\right)\left(R_2(T_0)+R_2(T)+4\right)+\left|\sum_{\substack{T_0 < \Im(\rho) \le T \\ 0>\Re(\rho)>b+1}} 							1+\sum_{\substack{T_0 < \Im(\rho)\le T \\ b-1 \ge \Re(\rho)>b}}(\Re(\rho)-b)\right|.
				\end{aligned}
		\end{equation}
		Since all the trivial zeros are of the form $s=-\frac{l+\mu_j}{\lambda_j}$, where $l=0, 1, 2,\ldots$ and $j \in [1,f]$, we have
		\begin{equation}
		\label{trivialEstimate}
				\begin{aligned}
						& \left|\sum_{\substack{T_0 < \Im(\rho) \le T \\ 0>\Re(\rho)>b+1}} 1+\sum_{\substack{T_0 < \Im(\rho) \le T \\ b+1 \ge 										\Re(\rho)>b}} (\Re(\rho)- b) \right| \\
						& \quad< f\cdot\left(\left|(b+1)\max_j\{\lambda_j\}+\min_j\{\Re(\mu_j)\}\right| \right. \\
						& \quad\quad \left. \vphantom{\left|(b+1)\max_j\{\lambda_j\}+\min_j\{\Re(\mu_j)\}\right|}-b\max_j\{\lambda_j\}+(b+1)\min_j\{\lambda_j\}-							\min_j\{\Re(\mu_j)\}+\max_j\{\Re(\mu_j)\}\right).
				\end{aligned}
		\end{equation}
		The claim follows from the inequalities \eqref{NIntegral} and \eqref{trivialEstimate}.			
\end{proof}

Next we prove the main result by estimating the term $R_\mathcal{L}(T_0,T)$. We want that it holds that 
$$
		\left|R_\mathcal{L}(T_0,T)\right|\le c_{\mathcal{L},1}\log{T}+c_{\mathcal{L},2}(T_0)+\frac{c_{\mathcal{L},3}(T_0)}{T},
$$ 
where the terms $c_{\mathcal{L},j}(T_0)$ are real numbers which depend on the function $\mathcal{L}$ and the number $T_0$ and the real number $c_{\mathcal{L},1}$ depends only on the function $\mathcal{L}$. To shorten our notation we define that
\begin{equation*}
\begin{aligned}
		h_{\mathcal{L},1}&= (1-\alpha)\left(\max\bigg\{\frac{5}{2}\log{|\lambda Q^2|}, \left(\frac{1}{2}-a+2R\right)\log{|\lambda Q^2|} \bigg\}+d_{\mathcal{L}}\left(-\frac{3}				{2}+4R\right)\right) \\
		& \quad+\alpha\left(\left(\frac{5}{2}d_{\mathcal{L}}+1\right)\log{2}+k\log{3}+ \max\bigg\{0, \frac{5}{2}\log{|\lambda Q^2|}+\frac{5\sqrt{5}}{2}d_{\mathcal{L}}+|			\Im(\mu)| \bigg\}\right) 
\end{aligned}
\end{equation*}
and
\begin{equation*}
		h_{\mathcal{L},2}=(1-\alpha)\left(d_{\mathcal{L}}\left(\frac{1}{2}-a+2R\right)(a+2R)+(a+2R)\left|\frac{\Im(\mu)}{2} \right|\right)
\end{equation*}
where the number $\alpha \in \{0,1\}$. If the sum $h_{\mathcal{L},1}+\frac{h_{\mathcal{L},2}}{T_0-2R}$ is bigger for $\alpha=0$ than $\alpha=1$ then $\alpha=0$. Otherwise $\alpha=1$. Using this notation we obtain the main result:
\begin{lause}
\label{mainResult}
		Suppose that $T_0$ and $T$ satisfy the same conditions as in Lemma \ref{N+}. Then we have
		\begin{equation*}
				|R_\mathcal{L}(T_0,T)| \le c_{\mathcal{L},1}(T_0)\log{T}+c_{\mathcal{L},2}(T_0)+\frac{c_{\mathcal{L},3}(T_0)}{T},
		\end{equation*}
		where			
		\begin{equation*}
				\begin{aligned}
						c_{\mathcal{L},1}&=\frac{1}{2\pi}\left(-7\frac{d_{\mathcal{L}}}{2}(2b+1)+2\left|-d_{\mathcal{L}}b+\frac{\Im(\mu)i}{2} \right| 									+2d_{\mathcal{L}}+S(1,e,b)\right) \\
						& \quad+ \frac{1}{\log 2}\left(R-\frac{1}{2}\right)d_{\mathcal{L}}\left(\frac{1}{2}-a+2R\right),
				\end{aligned}
		\end{equation*}
		\begin{equation*}
				\begin{aligned}
						c_{\mathcal{L},2}(T_0)&=\frac{d_{\mathcal{L}}}{2 \pi}T_0 \log \frac{T_0}{e}+\frac{T_0}{2\pi}\left|\log(\lambda Q^2)\right|+\frac{\pi}								{3\log{2}}+4R-2 \\
						& \quad +\frac{3d_{\mathcal{L}}(b^2+b)}{2\pi T_0}+ f\cdot\left(\left|(b+1)\max_j\{\lambda_j\}+\min_j\{\Re(\mu_j)\}\right| \right. \\	
						& \quad \left. \vphantom{\left|(b+1)\max_j\{\lambda_j\}+\min_j\{\Re(\mu_j)\}\right|}-b\max_j\{\lambda_j\}+(b+1)\min_j\{\lambda_j\}-								\min_j\{\Re(\mu_j)\}+\max_j\{\Re(\mu_j)\}\right) \\
						& \quad +\frac{1}{2\pi}R_1(T_0,1,b)+\left(R-\frac{1}{2}\right)\left(R_2(T_0)+d_{\mathcal{L}}\left(\frac{1}{2}-a+2R\right)\right) \\
						& \quad+\frac{1}{\log{2}}\left(R-\frac{1}{2}\right)\left(\log{\frac{a_1\pi^2}{6}}+h_{\mathcal{L},1} \right)
				\end{aligned}
		\end{equation*}
		and
		\begin{equation*}
				\begin{aligned}
						c_{\mathcal{L},3}(T_0)&=\frac{1}{\log{2}}\left(R-\frac{1}{2}\right)\frac{T_0}{T_0-2R}\left(V^*(2R+1)+h_{\mathcal{L},2}\right). 
				\end{aligned}
		\end{equation*}
\end{lause}

\begin{proof}
		Since $\left|1-\frac{\left(a+2R\right)i}{T_0-2R} \le1+\frac{a+2R}{T_0-2R}\right| $ and $\frac{1}{T-2R}\le \frac{T_0}{(T_0-2R)T}$, the claim follows from the definition of 			the term $R_\mathcal{L}(T_0,T)$.
\end{proof}

Using the main result we can prove a useful corollary.  If we know the number of up to height $T_0$, we can also estimate the the number of zeros up to height $T$. Let $\mathcal{N}_{\mathcal{L}}^{+}(t)$ and $\mathcal{N}_{\mathcal{L}}^{-}(t)$ denote the number of the non-trivial zeros of the function $\mathcal{L}$ for which $0\le \Im(\rho) \le t$ and $-t \le \Im(\rho) \le 0$ respectively. We also notice that by \cite{steuding} we have that $\mathcal{N}_{\mathcal{L}}^{\pm}(T) \sim \frac{d_\mathcal{L}}{2\pi}T\log{T}$ and thus the numbers $\mathcal{N}_{\mathcal{L}}^{\pm}(T_0)$ are finite. Using these properties we obtain the following corollary:

\begin{corollary}
		Suppose that $T_0$ and $T$ satisfy the same conditions as in Lemma \ref{N+}. Since for all positive real numbers $c$ it holds that $c\le c\log{T}$, by Theorem 					\ref{mainResult} we get
		\begin{equation*}
				|R_\mathcal{L}(T_0,T)+\mathcal{N}_{\mathcal{L}}^{\pm}(T_0)| \le c_{\mathcal{L},1}\log{T}+C_{\mathcal{L},2}(T_0)+\frac{c_{\mathcal{L},3}(T_0)}{T},
		\end{equation*}
		for example, when $C_{\mathcal{L},2}(T_0)=c_{\mathcal{L},2}(T_0)+\max\{\mathcal{N}_{\mathcal{L}}^{+}(T_0),\mathcal{N}_{\mathcal{L}}^{-}(T_0)\}$.
\end{corollary}

Furthermore, we also would like to note one interesting and useful result. For the number of zeros in the interval $(T,2T]$ we obtain a formula with the error term where coefficients of the terms $\log{T}, 1$ and $\frac{1}{T}$ don't depend on the number $T$.

\begin{remark}
\label{remarkc}
		Using similar methods as in the proof of Lemma \ref{N+} and Theorem \ref{mainResult} we can prove that
		\begin{equation*}
				\left | \mathcal{N}_\mathcal{L}^{\pm}(T,2T)-\frac{d_\mathcal{L}}{2\pi}T\log{\frac{4T}{e}}-\frac{T}{2\pi}\log({\lambda Q^2}) \right  | \le c_1 \log{T}+c_2+					\frac{c_3}{T},
		\end{equation*}
		where
		\begin{equation*}
				c_1 =\frac{d_{\mathcal{L}}}{\log 2}\left(2R-1\right)\left(\frac{1}{2}-a+2R\right),
		\end{equation*}
		\begin{equation*}
				\begin{aligned}
						c_2 &=\frac{\log{2}}{2\pi}\left(-\frac{7d_\mathcal{L}}{2}(2b+1)+2\left|-d_\mathcal{L}b+\frac{\Im(\mu)i}{2}\right|+2d_\mathcal{L}\right)+							\frac{S(1,2,b)}{2\pi}+\frac{2\pi}{3\log{2}} \\
						& \quad +4R-2+\left(2R-1\right)\left(\frac{3d_{\mathcal{L}}}{2}\left(\frac{1}{2}-a+2R\right)+\frac{1}{\log{2}}\left(\log{\frac{a_1\pi^2}{6}}+							+h_{\mathcal{L},1}\right)\right)
				\end{aligned}
		\end{equation*}
		and
		\begin{equation*}
				c_3 =\frac{3d_\mathcal{L}(b^2+b)}{4\pi}+\frac{3T_0}{2(T_0-2R)\log{2}}\left(R-\frac{1}{2}\right)\left(V^*(2R+1)+h_{\mathcal{L},2}\right).
		\end{equation*}
\end{remark}

\section{Example: $L$-function associated with a holomorphic newform}
\label{numericalExample}

In this section we give examples of the values of the terms $c_{\mathcal{L},1}$, $c_{\mathcal{L},j}(T_0)$ and $c_j$ which are defined in Section \ref{sectionMainResult}. Since we have estimated these terms for a general set which contains $L$-functions other than the Riemann zeta function and Dirichlet $L$-functions, the estimates of the term $c_{\mathcal{L},1}$ and $c_{\mathcal{L},j}(T_0)$ for these functions are not as strong as previous estimates, see \cite{backlund} and \cite{trudgian2}.

Let $g$ be a newform of even weight $\kappa$ for some congruence subgroup
\begin{equation*}
		\Gamma_0(N)=\left\{ \begin{pmatrix} A & B\\ C & D \end{pmatrix} \in SL_2(\mathbb{Z}) :  C \equiv 0 \mod N \right\},
\end{equation*}
where $N$ is a positive integer and $SL_2(\mathbb{Z})$ is a set of $2\times 2$ matrices with integer entries and which determinant is $1$. We also assume that for $z\in \mathbb{H}$ the function $g$ has a Fourier expansion
\begin{equation*}
		g(z)=\sum\limits_{n=1}^\infty c(n)\exp(2\pi i n z).
\end{equation*}
We define 
\begin{equation*}
		\mathcal{L}(s)=\sum\limits_{n=1}^\infty \frac{a(n)}{n^s},
\end{equation*}
where $a(n)=c(n)n^{\frac{1-\kappa}{2}}$. The function $\mathcal{L}(s)$ satisfies the equation
\begin{equation*}
		\Lambda_{\mathcal{L}}(s)=\mathcal{L}(s)\left(\frac{\sqrt{N}}{2\pi}\right)^s\Gamma\left(s+\frac{\kappa-1}{2}\right),
\end{equation*}
where
\begin{equation*}
		\Lambda_{\mathcal{L}}(s)=i^{\kappa}\Lambda_{\mathcal{L}}(1-s).
\end{equation*}
Hence, we can choose 
\begin{equation*}
		f=1, Q=\frac{\sqrt{N}}{2\pi}, \lambda_j=1, \omega=i^\kappa \text{ and } \mu_j=\frac{\kappa-1}{2}. 
\end{equation*}
Thus we also have  $d_\mathcal{L}=2$, $\lambda=1$ and $\mu=4-2\kappa$. We also have $k=0$. By Deligne \cite{deligne1,deligne2} $|a(n)|\le 1$ and we can choose that $a_1=1$, $a=3$ and $b=-4$. By Theorem \ref{mainResult} we have $T_0\ge 15+\kappa$ and
\begin{equation*}
		\begin{aligned}
				c_{\mathcal{L},1}&= \frac{299}{2\log 2}+\frac{1}{2\pi}\left(3\kappa^2-2\kappa+\frac{217}{3} \right. \\
								       & \quad \left. +\frac{1}{12}\left(\sec^2\left(\frac{\arg\big(4+(\kappa+1)i\big)}{2}\right)+															\sec^2\left(\frac{\arg\big(3+(\kappa+1)i\big)}{2}\right)\right)\right),
		\end{aligned}
\end{equation*}
\begin{equation*}
		\begin{aligned}
				c_{\mathcal{L},2}(15+\kappa)&=\frac{1}{\pi}(15+\kappa) \log \frac{15+\kappa}{e}+\frac{15+\kappa}{2\pi}\left|\log{\frac{N}{4\pi^2}}\right|+\frac{\pi}						{3\log{2}}+\frac{353}{2} \\
				& \quad +\frac{36}{\pi (15+\kappa)}+\left|\frac{\kappa-7}{2}\right |+\frac{72}{2\pi(15+\kappa)} \\
				& \quad -\frac{\log{(15+\kappa)}}{2\pi}\left(\vphantom{\sec^2\left(\frac{\arg\big(3+(\kappa+1)i\big)}{2}\right)}\frac{9\kappa^2-6\kappa+217}{3}+							\right.\\
				& \quad \left.+\frac{1}{12}\left(\sec^2\left(\frac{\arg\big(4+(\kappa+1)i\big)}{2}\right)+\sec^2\left(\frac{\arg\big(3+(\kappa+1)i\big)}{2}\right)\right)						\right) \\
				& \quad+\frac{13}{12(1+\kappa)\log{2}}\left(9\kappa^2-6\kappa+10+\frac{1}{2}\sec^2\left(\frac{\arg\big(-17+(\kappa+1)i\big)}{2}\right)\right) \\
				& \quad+\frac{299}{2\log{2}}\left(\log{(30+2\kappa)}+\left|1-\frac{17i}{\kappa+1}\right|\right) \\
				& \quad+\frac{13}{2\log{2}}\left(2\log{\frac{\pi^2}{6}}+2\max\bigg\{\frac{5}{2}\log{\frac{N}{4\pi^2}},\frac{23}{2}\log{\frac{N}{4\pi^2}} 							\bigg\}+83\right)
		\end{aligned}
\end{equation*}
and
\begin{equation*}
		\begin{aligned}
				c_{\mathcal{L},3}(15+\kappa)&=\frac{13}{12\log{2}}\frac{15+\kappa}{1+\kappa}\left(\vphantom{\sec^2\left(\frac{\arg\big(-17+(\kappa+1)i\big)}{2}					\right)}9\kappa^2-6\kappa+2356\right. \\
				& \quad \left.+\frac{1}{2}\sec^2\left(\frac{\arg\big(-17+(\kappa+1)i\big)}{2}\right)\right).
		\end{aligned}
\end{equation*}
Furthermore, by Remark \ref{remarkc} we also have
\begin{equation*}
				c_1 =\frac{299}{\log{2}},
\end{equation*}

\begin{equation*}
		\begin{aligned}
				c_2 &=\frac{2\pi}{3\log{2}} +923+\frac{\log{2}}{6\pi}\left(\vphantom{\sec^2\left(\frac{\arg\big(4+(\kappa+1)i\big)}{2}\right)}9\kappa^2-6\kappa+217 					\right. \\
				& \quad \left.+\frac{1}{4}\left(\sec^2\left(\frac{\arg\big(4+(\kappa+1)i\big)}{2}\right)+\sec^2\left(\frac{\arg\big(3+(\kappa+1)i\big)}{2}\right)\right)						\right)\\
				& \quad +\frac{13}{\log{2}}\left(\log{\frac{\pi^2}{6}}+\max\bigg\{\frac{5}{2}\log{\frac{N}{4\pi^2}}, \frac{23}{2}\log{\frac{N}									{4\pi^2}} \bigg\}+53\right)
		\end{aligned}
\end{equation*}
and
\begin{equation*}
		\begin{aligned}
				c_3 &=\frac{18}{\pi}+\frac{13(15+\kappa)(17+3\kappa)}{4(1+\kappa)(8+\kappa)\log{2}}\left(\frac{1}{6(1+\kappa)}\left(\vphantom{\frac{1}{2}						\sec^2\left(\frac{\arg\big(-17+(\kappa+1)i\big)}{2}\right)}9\kappa^2-6\kappa+10 \right.\right. \\
				& \quad \left.\left.+\frac{1}{2}\sec^2\left(\frac{\arg\big(-17+(\kappa+1)i\big)}{2}\right)\right)+391\right). 
		\end{aligned}
\end{equation*}
We can see different values of the ceiling function of the numbers $c_{\mathcal{L},1}$, $c_{\mathcal{L},j}(T_0)$ and $c_j$ from Table \ref{table}.

\begin{table}[!h]
\caption{Different values of the terms $c_{\mathcal{L},1}$, $c_{\mathcal{L},j}(T_0)$ and $c_j$}
\label{table}
\begin{center}
\begin{spreadtab}{{tabular}{*{9}{>{$}r<{$}}}}
\toprule
@N & @\kappa & @T_0 & @c_{\mathcal{L},1}  & @c_{\mathcal{L},2}(T_0) & @c_{\mathcal{L},3}(T_0) & @c_1 & @c_2 & @c_3\\
\midrule
1
& 12
& \STcopy{>1,v24}{!15+b2} 
& 293
& 1945
& 11637
& 432
& 1811
& 10506
 \\
1  & 34 & & 769 & 415 & 27478 & 432 & 2141 & 8183\\
1  & 36 & & 835 & 172 & 29742 & 432 & 2187 & 8133 \\
1  & 38 & & 905 & -91 & 32127 & 432 & 2235 & 8092\\
1  & 40 & & 979 & -374 & 34631 & 432 & 2286 & 8060\\
1  & 50 & & 1405 & -2087 & 48918 & 432& 2582 & 7983\\
2  & 8 & & 256 & 2112 &11554 & 432 & 1817 & 12323\\
2 &10 & & 272& 2040  & 11375  & 432 & 1829 & 11247\\
11 &2 & & 229  & 2941 & 21661 & 432 & 1879 & 24239\\
11 &10 &  & 272 & 2113 &  11375 & 432 & 1909 & 11247\\
11 &12 & &  293 & 2047 & 11637 & 432 & 1923 &10506 \\
11 &36 & & 835 & 265 & 29742 & 432 & 2299 & 8133\\
11 &38 & & 905 & 1 & 32127 & 432 & 2347 & 8092\\
11 & 40 & & 979 & -282 & 34631 & 432 & 2399 & 8060\\
21 &6 & & 243 & 2314 &12460 & 432 & 1919 & 14017\\
21 &8 & & 256 & 2214 &11554 & 432 & 1928 & 12323 \\
40 &2 & & 229 & 3000 & 21661 & 432 & 1942 & 24239\\
40 &6 &  & 243 & 2345 &12460 & 432 & 1951 & 14017 \\
40 &36 & & 835 & 317 & 29742 & 432 & 2362 & 8133\\
40 &38 & & 905 & 53 & 32127 & 432 & 2410 & 8092\\
63 &36 & & 835 & 419 & 29742 & 432 & 2460 & 8133\\
63 &38 & & 905 & 155 & 32127 & 432 & 2508 & 8092\\
64 &36 & & 835 & 422 & 29742 & 432 & 2463 & 8133\\
64 &38 & & 905 & 159 & 32127 & 432 & 2512 & 8092\\
64 &40 & & 979 & -125 & 34631 & 432 & 2563 & 8060 \\
\bottomrule
\end{spreadtab}
\end{center}
\end{table}

\clearpage
\bibliographystyle{abbrv}
\bibliography{Number_of_zeros_Selberg_class-Copy}

\end{document}